\documentclass[onefignum,onetabnum]{siamonline250211}

\usepackage{amsmath, bbm, bm}
\usepackage{lipsum}
\usepackage{amsfonts}
\usepackage{graphicx}
\usepackage{epstopdf}
\usepackage{mathtools}
\usepackage{amssymb}
\usepackage{accents}
\usepackage{algorithm}
\usepackage{algpseudocode}
\usepackage{tabto}
\usepackage{aliascnt}
\usepackage{cleveref}
\usepackage[dvipsnames]{xcolor}
\usepackage[textsize=tiny]{todonotes}

\ifpdf
  \DeclareGraphicsExtensions{.eps,.pdf,.png,.jpg}
\else
  \DeclareGraphicsExtensions{.eps}
\fi

\usepackage{enumitem}
\setlist[enumerate]{leftmargin=.5in}
\setlist[itemize]{leftmargin=.5in}

\newcommand{\la}{\langle}
\newcommand{\ra}{\rangle}
\newcommand{\half}{\frac{1}{2}}

\makeatletter
\newcommand{\uprighttheorem}[1]{%
  \expandafter\let\csname old#1\expandafter\endcsname\csname #1\endcsname
  \expandafter\let\csname endold#1\expandafter\endcsname\csname end#1\endcsname
  \expandafter\renewenvironment\expandafter{#1}%
    {\csname old#1\endcsname\normalfont}%
    {\csname endold#1\endcsname}%
}
\makeatother

\newaliascnt{remark}{theorem}
\newtheorem{remark}[remark]{Remark}
\aliascntresetthe{remark}
\uprighttheorem{remark}

\newaliascnt{condition}{theorem}
\newtheorem{condition}[condition]{Condition}
\aliascntresetthe{condition}

\newaliascnt{fact}{theorem}
\newtheorem{fact}[fact]{Fact}
\aliascntresetthe{fact}
\uprighttheorem{fact}

\crefname{theorem}{Theorem}{Theorems}
\crefname{hypothesis}{Hypothesis}{Hypotheses}
\crefname{remark}{Remark}{Remarks}
\Crefname{remark}{Remark}{Remarks}
\crefname{condition}{Condition}{Conditions}
\Crefname{condition}{Condition}{Conditions}
\crefname{fact}{Fact}{Facts}
\Crefname{subsection}{Section}{Sections}

\headers{Derivative-informed dimension reduction of SDEs}{J. Zou, H. C. Lie, Y. Marzouk}

\title{Derivative-Informed Dimension Reduction of Stochastic Differential Equations: A Path-Space Sensitivity Approach}

\author{Joanna Zou\thanks{Center for Computational Science \& Engineering, Massachusetts Institute of Technology, Cambridge, MA 
  (\email{jjzou@mit.edu}).}
\and Han Cheng Lie\thanks{Department of Mathematics, University of Potsdam, Potsdam, Germany 
  (\email{han.lie@uni-potsdam.de}).}
\and Youssef Marzouk\thanks{Center for Computational Science \& Engineering, Massachusetts Institute of Technology, Cambridge, MA 
  (\email{ymarz@mit.edu}).}}

\usepackage{amsopn}

\newcommand{\upd}{\textup{d}}

\definecolor{english}{rgb}{0.0, 0.5, 0.0}

\ifpdf
\hypersetup{
  pdftitle={Derivative-informed dimension reduction of SDEs},
  pdfauthor={J. Zou, H. C. Lie, Y. Marzouk}
}
\fi

\begin{document}

\maketitle

\begin{abstract}
Uncertainty in the drift of a stochastic differential equation (SDE) induces uncertainty in statistical observables defined as expected functionals on path space. We propose a derivative-informed method for identifying low-dimensional structure in the drift-to-observable map, with the aim of improving the tractability of uncertainty quantification. This approach can be viewed as a generalization of the active subspace method to infinite-dimensional input spaces and to path-space observables of SDEs. Representing uncertainty in the drift by random perturbations in
a Hilbert space, we construct a path-space sensitivity operator using first-order Fréchet derivatives of the observable with respect to the perturbation. We establish conditions under which the leading eigenfunctions of this operator correspond to directions of maximum mean squared sensitivity of the observable, thus enabling low-dimensional approximation with certified error bounds. We further introduce a martingale control variate of the Fr\'echet derivative that reduces error in sample-based estimators of the operator. Using a Galerkin scheme for numerical discretization, we show that reduced representations of the drift from the derivative-informed basis yield more accurate approximations of mean first hitting times of SDEs than the Karhunen--Loève expansion of the same truncation order. 
\end{abstract}

\begin{keywords}
Forward uncertainty quantification, surrogate operators, parameter-to-observable maps, active subspaces, stochastic dynamics, path-space statistics, mean first hitting times
\end{keywords}

\begin{MSCcodes}
65C30, 60H10, 60B11
\end{MSCcodes}

\section{Introduction}
\label{sec:intro}

Stochastic differential equations (SDEs) are used to represent stochastic processes in a wide array of domains, including physics and biology, systems engineering, mathematical finance, sampling, and diffusion-based generative modeling. In many of these settings, the drift function of the SDE is not known \textit{a priori} and must be inferred from a set of observations. Data-driven models of the drift function appear in the form of machine learning force fields for molecular dynamics \cite{Batzner2022,Musaelian2023}, surrogates of model discrepancy in control systems \cite{Psichogios1992,Johannink2019}, and neural network estimates of the score of an intractable probability distribution in diffusion-based sampling \cite{Song2021}. In each of these examples, models of the drift function are inferred from finite data and therefore subject to statistical error, which propagates to uncertainty in statistical properties of the modelled process. 

Both the tasks of model inference and uncertainty quantification of data-driven models of SDEs suffer from the curse of dimensionality: it is more difficult to infer an accurate model of the drift function and quantify the impact of model uncertainty when its parameters are high- or infinite-dimensional. To pose an example, in forward uncertainty quantification (UQ), one aims to characterize how uncertainty in the input drift function -- represented by a probability distribution of parameters of the drift function, $\mu$ -- influences uncertainty in an output observable quantity $F$ which depends nonlinearly on the drift function. This task amounts to solving for the pushforward distribution of the observable, $F_\sharp \mu$, which is typically not available in closed form and must be characterized by Monte Carlo simulation. Rather than incur the expense of simulating the SDE for every Monte Carlo trial, a common approach to reduce the cost of forward UQ is to learn a surrogate operator which efficiently approximates the pushforward map \cite{Sudret2017,Huang2025}. However, such operators are increasingly challenging to learn for input parameters in high dimensions.

Dimension reduction techniques are often employed to improve the tractability of forward UQ, sensitivity analysis, and model learning. In the literature on dimension reduction for parameters of SDEs, the drift function is generally assigned a specific functional form -- for instance, reduced-rank Gaussian processes \cite{Hostettler2018} or a low-rank tensor decomposition of the value function in stochastic optimal control \cite{Gorodetsky2015,Richter2021} -- which admits reduced-dimensional representations. However, these strategies reduce dimension by committing to a specific model class, which precludes the use of preexisting or alternative functional forms such as neural networks. Moreover, since the basis constructions are agnostic to variation in observable functionals of the solution of the SDE, such as path statistics or invariant statistics, truncation of the basis expansion form of the drift may lead to a loss of accuracy in the corresponding observable. 


In this paper, we propose a technique for constructing a reduced representation of the drift function in terms of a data-driven, observable-adapted basis which is suitable for forward UQ and sensitivity analysis of the observable of the SDE. Our approach corresponds to solving for the \textit{derivative-informed subspace} of parameters of the drift function from the spectral decomposition of a \textit{path-space sensitivity operator}, constructed from data on the first-order derivative of the observable with respect to the parameters. The derivative-informed subspace is also referred to as the active subspace in literature, which was first popularized by \cite{Constantine2020} and then generalized to vector-valued observables in \cite{Zahm2020} and to infinite-dimensional parameters and observables in \cite{Luo2025}. In the setting where the parameter is infinite-dimensional -- the drift function itself is the parameter -- the subspace approximation of the drift is a truncated expansion on a complete basis of a chosen function space. A key feature of derivative-informed subspaces is that under appropriate conditions, one can leverage a subspace Poincar\'e inequality to derive an upper bound to error in the subspace approximation of the observable.

We focus the presentation of our method on its application to \textit{path-space observables} of SDEs. These are quantities expressed as expectations with respect to the law of the SDE, which is a probability measure on the space of continuous paths. This type of observable appears in many forms, such as committor functions or mean first transition times which relate to reaction rates in molecular dynamics \cite{Lelievre2010,Zou2026,Du2026} and the cost functional in stochastic optimal control \cite{Hartmann2012,Lie2021,Nusken2021}. 
Our method is similar to Koopman operator approaches to dimension reduction in that both seek representations of the SDE coefficients that are adapted to a chosen class of observables. The Koopman operator characterizes the evolution of observables dependent on the process at deterministic times, such as one-point statistics or two-point correlation functions \cite{Zhang2021}, and can be represented with reduced-dimensional Galerkin or extended dynamic mode decomposition approximations which preserve the selected fixed-time statistics \cite{Williams2015,Klus2018,Klus2020,Peitz2019}. In contrast, the path-space observables considered here are defined expectations over path measures, where the length of paths, determined by the stopping time, is a random variable. In general, such quantities cannot be expressed as the action of a fixed-time Koopman operator on a state-space observable. Derivative-informed subspaces are a more general approach to dimension reduction in the sense that they can tailor the reduced subspace of drift functions to not only one-point or two-point statistics, but also path-dependent statistics defined by non-deterministic stopping times.

It is important to distinguish our approach, as well as those reviewed above, as techniques for \textit{parameter space reduction}, in which the drift function with reduced rank or fewer degrees of freedom continues to operate on full-dimensional state space. These methods are distinct from those for model order reduction (MOR) \cite{Mezic2005} such as the recursive polynomial chaos (RPC) method \cite{Bal2026}, in which the state dimension is reduced and one constructs Galerkin projections or surrogate models of the SDE coefficients on the low-dimensional domain. Our approach to drift approximation is agnostic to the dimension of the state space and is even amenable to approximating drift functionals on infinite-dimensional Hilbert spaces; therefore, it can be used in conjunction with MOR techniques for SDEs, although this direction is not explored further in this paper.  

\subsection{Summary of contributions}

We define a derivative-informed basis of Hilbert spaces of drift functions informed by the sensitivity of a particular path-space observable of the SDE. We develop sufficient conditions under which it forms an orthonormal basis of the Cameron-Martin space associated with the Gaussian measure on the Hilbert space, thus enabling finite-dimensional approximations via basis truncation. We formulate the derivative-informed subspace in the most general setting of infinite-dimensional spaces of drift functions, such that our results can readily extend to the setting of finite-dimensional parameterization of the drift function. Our contributions can be summarized as follows: 

\begin{itemize}
    \item In \Cref{lem:controlspace}, we specify that the drift perturbations must belong to a Hilbert space of bounded functions in order for the SDE to have a unique weak solution. A choice of Hilbert space meeting this condition is a Sobolev space which is continuously embedded in a H\"older space of continuous, bounded functions. On this space, one can define Gaussian measures with covariance given by a fractional Laplacian operator with particular requirements on the smoothness parameter according to \Cref{fact:measure}, so that the Cameron-Martin space is continuously embedded in the Sobolev space. 
    \item In \Cref{prop:condf}, we discuss sufficient conditions on the path-space observable for it to meet \Cref{cond:f}. Under these conditions, the eigenfunctions comprising the derivative-informed basis form a complete orthonormal basis of the Cameron-Martin space (\Cref{lem:traceclass}) and the eigenvalues quantify the sensitivity of the observable along each eigendirection (\Cref{lem:directderiv}). As a result, one can construct an analytical optimal approximation of the path-space observable with respect to the derivative-informed subspace (\Cref{prop:optapprox}) and a certified error bound on the approximation (\Cref{prop:di}). 
    \item In \Cref{sec:discretization}, we introduce discrete representations of the infinite-dimensional operators associated with derivative-informed subspaces. We do so by constructing Galerkin projections of these operators onto a finite element subspace of the Sobolev space. We discuss the practical computation of the derivative-informed subspace using Monte Carlo estimation of the Fr\'echet derivative and the associated sensitivity operator. 
    \item In \Cref{sec:controlvariate}, we propose a control variate formulation of the Fr\'echet derivative of the path-space observable which leads to reduced-variance estimates of the sensitivity operator. In particular, we use the local martingale appearing in the Fr\'echet derivative formula in \Cref{lem:frechet} as the control variate, which offers the advantage that its first moment is known and it is available at no additional computational cost. 
\end{itemize}

\medskip 

An outline of the paper is as follows. \Cref{sec:background} introduces the necessary background for our work, including the definition of the SDE and path-space observables considered (\Cref{sec:sde}), the formulation of Gaussian measures on Hilbert spaces (\Cref{sec:hilbert}), and approaches to dimension reduction on Hilbert spaces (\Cref{sec:dimred}). \Cref{sec:contribution} discusses our main results as summarized above. \Cref{sec:numerics} presents numerical experiments performed using FEniCS \cite{Alnaes2015} and hIPPYlib \cite{Villa2021} Python libraries to estimate the pushforward distribution of mean first hitting times of SDEs arising from Gaussian measures on drift perturbations. 

 
\section{Background}
\label{sec:background}

\subsection{Perturbed stochastic differential equations}
\label{sec:sde}

Consider a dynamical system whose state $X^u(\omega)=(X^u_t(\omega))_{t \geq 0}$ is a solution to an It$\hat{\text{o}}$ stochastic differential equation (SDE) on some filtered probability space $(\Omega, \mathcal{F}, (\mathcal{F}_t )_{t \geq 0}, P)$,
    \begin{equation}
        \label{eq:sde}
        \upd X^u_t = (b + \sigma u)(X^u_t) \, \upd t + \sigma(X^u_t) \upd W_t, \qquad X^u_0 = x \ ,
    \end{equation}
given an initial condition $x \in \mathbb{R}^n$ and standard Brownian motion $W:\Omega \to C([0,T], \mathbb{R}^n)$ on $(\mathcal{F}_t )_{t \geq 0}$. The solution $X^u: \Omega \to C([0,T], \mathbb{R}^n)$ is a Markov process on $\mathbb{R}^n$ over the time interval $[0,T]$. We assume the following conditions on the drift function $b$ and diffusion coefficient $\sigma$. 

\begin{condition}[Conditions on the SDE coefficients]
    \label{cond:driftdiff}
    $b: \mathbb{R}^n \to \mathbb{R}^n$ and $\sigma: \mathbb{R}^n \to \mathbb{R}^{n \times n}$ are time-invariant and satisfy conditions for $\upd X_t = b(X_t) \upd t + \sigma(X_t) \upd W_t, X_0 = x$ to have a unique $(\mathcal{F}_t )_{t \geq 0}$-adapted strong solution $(X_t)_{t \geq 0}$. Moreover, there exists a constant $c_0 > 0$ 
    such that for all $x,y \in \mathbb{R}^n$,
        \[
        0 \leq x^\top \big( \sigma(y) \sigma^\top(y)\big)^{+} x \leq c_0^{-2} \|x\|^2 \ .
        \]
\end{condition}
%

We utilize \eqref{eq:sde} as a form of SDE in which uncertainty in the drift $b$ is modeled by an additive perturbation field $u: \mathbb{R}^n \to \mathbb{R}^n$ scaled by $\sigma$.
Given any $u$ for which $X^u$ is unique in law, let $\mathbb{E}^u[\cdot] = \mathbb{E}_{\mathbb{P}^u}[\cdot]$ denote the expectation operator with respect to $\mathbb{P}^u \coloneqq P \circ (X^u)^{-1}$. In the case of $u=0$, the SDE in \eqref{eq:sde} has a unique strong solution, denoted $X^0$ with law $\mathbb{P}^0 \coloneqq P \circ (X^0)^{-1}$, as a result of the linear growth and Lipschitz continuity assumptions of \cref{cond:driftdiff} \cite[Thm. 5.2.1]{Oksendal2013}. We designate $\mathbb{U}$ to be a set of admissible perturbation fields and defer further specification of its properties to \Cref{sec:hilbertdrift}.

\subsubsection{Path space observables of SDEs}
\label{sec:observable}

Consider a \textit{path-space observable} $F: \mathbb{U} \to \mathbb{R}$ to be a functional of the drift perturbation, defined as the expectation of a path functional $\varphi_\tau : C([0,\infty),\mathbb{R}^n) \to \mathcal{Y}$ with respect to the law of the perturbed SDE, e.g.,
    \begin{equation}
        \label{eq:observable}
        F(u) \coloneqq \mathbb{E}^u[\varphi_\tau] \ 
    \end{equation}
for some stopping time $\tau > 0$. $F$ may also be understood as a particular choice of parameter-to-observable map \cite{Huang2025}, where the parameter of the SDE model is the drift perturbation $u$.

A canonical example of a path-space observable is the mean first hitting time (MFHT) to an open set $D \subset \mathbb{R}^n$ with positive Lebesgue measure, which we refer to as the hitting domain. For a stochastic process $Y = (Y_t)_{t \geq 0}$, the first hitting time is the map  $\tau_\infty: C([0,\infty), \mathbb{R}^n) \to [0,\infty)$ defined as
    \begin{equation}
    \label{eq:stopping}
        \tau_\infty(Y) \coloneqq \inf \{ t \geq 0: Y_t \in D \} \ .
    \end{equation}
with the convention $\inf \emptyset = \infty$. In general, $\tau_\infty$ is an unbounded random variable. However, for ease of exposition, we assume throughout that the path-space observable is defined by a bounded random stopping time as follows.

\begin{condition}
    \label{cond:stoptime}
    The path-space observable $F$ in \eqref{eq:observable} is defined by the functional $\varphi_{\tau}:C([0,T],\mathbb{R}^n) \to \mathcal{Y}$, where $\tau \coloneqq \tau_\infty(Y) \wedge T$ is the truncated stopping time given some fixed maximum time $T > 0$.
\end{condition}
Although boundedness of the stopping time is a stronger assumption than what is required for the results that follow, it is imposed here as a convenient sufficient condition. This truncation is natural in practice, since numerical simulations are run to a finite time $T$ determined by a fixed computational budget.

\subsection{Gaussian measures on Hilbert spaces}
\label{sec:hilbert}

We introduce the construction of Gaussian measures on infinite-dimensional Hilbert spaces. A measure $\mu$ on a Hilbert space $\mathbb{U}$ is Gaussian if for any continuous linear functional $\ell:\mathbb{U} \to \mathbb{R}$, the measure $\ell_\sharp\mu$ is Gaussian \cite[Def. 2.2.1]{Bogachev1998}. In other words, an element $u \in \mathbb{U}$ is a Gaussian process with law $\mu = \mathcal{N}(m_0,\mathcal{C})$ if $\ell(u) \sim \mathcal{N}(m_\ell, \sigma^2_\ell)$, where $m_\ell=\ell(m_0)$ and $\sigma^2_\ell = \ell(\mathcal{C} \ell)$ \cite{Stuart2010}. The series-based construction of the Gaussian measure is discussed further in \Cref{rmk:kl}. 

The \textit{Cameron-Martin space} $\mathcal{E} \subseteq \mathbb{U}$ associated with the covariance operator $\mathcal{C}$ is a Hilbert space of all $\mathcal{C}^{\half}$-transformations of elements of $\mathbb{U}$, e.g.,
\[
\mathcal{E} = \text{Im}(\mathcal{C}^{\half}) = \{\mathcal{C}^{\half} h : h \in \mathbb{U}\} \, ,
\]
equipped with the inner product $\la w, w' \ra_\mathcal{E} \coloneqq \la \mathcal{C}^{-\half} w, \mathcal{C}^{-\half} w' \ra_\mathbb{U}$ for $w, w' \in \mathcal{E}$ \cite[Lemma 6.10]{Stuart2010}. Equivalent definitions of $\mathcal{E}$ are in \cite[pg. 44]{Bogachev1998} and \cite[Thm. 6.11]{Stuart2010}. One can view $\mathcal{C}^{\half}: \mathbb{U} \to \mathcal{E}$ as an isometric isomorphism onto $\mathcal{E}$, with the inverse operator $\mathcal{C}^{-\half}: \mathcal{E} \to \mathbb{U}$. When the covariance $\mathcal{C}: \mathbb{U} \to \mathbb{U}$ is strictly positive definite, then the precision operator $\mathcal{C}^{-1}: \text{Im}(\mathcal{C}) \to \mathbb{U}$ is well-defined.
Note that $\mathcal{C}^{-1}$ is an unbounded operator with domain $\text{Im}(\mathcal{C}) \subset \mathcal{E}$, since for any $h \in \mathbb{U}$, $\mathcal{C}h = \mathcal{C}^\half (\mathcal{C}^\half h) \in \text{Im}(\mathcal{C}^\half) = \mathcal{E}$.

\begin{remark}
    Since $\mu(\mathcal{E}) = 0$ \cite[Thm. 2.4.7]{Bogachev1998}, an element $u \in \mathbb{U}$ is almost surely not in $\mathcal{E}$.
    Since $\mathcal{E}$ is continuously embedded in $\mathbb{U}$, an operator defined on $\mathbb{U}$ is also defined on $\mathcal{E}$. In contrast, an operator defined on $\mathcal{E}$ need not extend to all of $\mathbb{U}$.
\end{remark}

An example of a Gaussian measure that we use is $\mu = \mathcal{N}(0,\mathcal{A}^{-\nu})$, where the covariance is given by a fractional Laplacian operator on a bounded open set $\mathcal{X} \subset \mathbb{R}^n$ with $\mathcal{A} \coloneqq -\Delta$ and $\nu > 0$. More generally, for any Laplacian-like operator $\mathcal{B}$ on $\mathbb{U} \subset L^2(\mathcal{X})$ satisfying the assumptions in  \cite[Assump.~2.9]{Stuart2010}, $\mathcal{B}$ has the eigenpairs $\{(\lambda_k, \psi_k)\}_{k \geq 0}$ where $\{\psi_k\}_{k \geq 0}$ is an orthonormal basis of $\mathbb{U}$ and $\lambda_k \asymp |k|^2$. For $\nu \in \mathbb{R}$, the fractional power $\mathcal{B}^\nu$ can be defined spectrally as
\begin{equation}
\label{eq:Bnu}
\mathcal{B}^\nu u = \sum_{k \geq 0} \lambda_k^\nu \, \langle u, \psi_k \rangle_{\mathbb{U}} \, \psi_k \, .
\end{equation}
This expansion also enables the definition of the \emph{Hilbert scale} generated by $\mathcal{B}$. For $s \in \mathbb{R}$,
\begin{equation}
\label{eq:hilbertscale}
\mathbb{U}^s \coloneqq \Bigl\{ u : \|u\|_{\mathbb{U}^s}^2 \coloneqq \sum_{k \geq 0} \lambda_k^s \, |\langle u, \psi_k \rangle_{\mathbb{U}}|^2 < \infty \Bigr\} 
\end{equation}
is a separable Hilbert space satisfying $\mathbb{U}^s \subseteq \mathbb{U}$ for $s \geq 0$.
    


\subsection{Dimension reduction on Hilbert spaces}
\label{sec:dimred}

The series expansion form of a stochastic process lends itself to dimension reduction via basis truncation. In the following, we review two spectral methods to dimension reduction of stochastic processes, the Karhunen-Lo\`eve expansion and derivative-informed subspace.

\subsubsection{Karhunen-Lo\`eve expansion}
\label{sec:kl}

A well-known approach to the dimension reduction of second-order stochastic processes utilizes the Karhunen-Lo\`eve (KL) expansion, the infinite-dimensional analogue of principal component analysis (PCA) for vector-valued parameters. The KL expansion expresses a stochastic process $u \in \mathbb{U}$ as a basis expansion on a particular orthonormal basis of $\mathbb{U}$, given by the spectral decomposition of the covariance operator associated with the stochastic process. 

\begin{definition}[Karhunen-Lo\`eve (KL) basis]
\label{def:klbasis}
The Karhunen-Lo\`eve (KL) basis $\{\psi^\textup{KL}_i\}_{i \geq 0}$ is the set of eigenfunctions of $\mathcal{C}$, solving
    \begin{equation}
        \label{eq:eig_pca}
        \mathcal{C} \psi^{\textup{KL}}_i = \lambda^{\textup{KL}}_i \psi^{\textup{KL}}_i.
    \end{equation}
\end{definition}
Since the covariance operator $\mathcal{C}$ is self-adjoint, trace-class \cite[Thm. 6.6]{Stuart2010}, and is assumed to be strictly positive definite, the KL basis is an orthonormal basis in $\mathbb{U}$ and the eigenvalues satisfy $\lambda^{\textup{KL}}_0 \geq \lambda^{\textup{KL}}_1 \geq \dots > 0$ with $\lim_{i \to \infty} \lambda_i^{\textup{KL}} = 0$.

The truncated KL expansion of $u$ corresponds to its projection onto the first $r$ eigenfunctions of $\mathcal{C}$. Define the projection operator $P^{\textup{KL}}_r: \mathbb{U} \to \mathbb{U}$ as 
\[
P^{\textup{KL}}_r \coloneqq \sum_{i=0}^r \psi^{\textup{KL}}_i \otimes \psi^{\textup{KL}}_i \ ,
\]
where $\psi \otimes \psi \coloneqq \psi \la \psi, \cdot \ra_\mathbb{U} $. Then the reduced KL approximation of any $u \in \mathbb{U}$ is given by 
\begin{equation}
    \label{eq:proj_pca}
    u^\textup{KL}_r = m_0 + P^\textup{KL}_r (u - m_0) = m_0 + \sum_{i=0}^r \psi_i^{\textup{KL}} \la \psi_i^{\textup{KL}}, u - m_0 \ra_\mathbb{U}  . 
\end{equation}

\begin{remark}
\label{rmk:kl}
    The truncated KL expansion is commonly expressed in terms of a random variable $\xi_i$  \cite{Wang2008}, defined as $\xi_i(\omega) \coloneqq \frac{1}{\sqrt{\lambda_i^{\textup{KL}}}} \la \psi_i^{\textup{KL}}, u(\omega) - m_0 \ra_\mathbb{U}$ for a given realization $\omega$ in the sample space $\Omega$. Note that $\mathbb{E}[\xi_i] = 0$ and $\mathbb{E}[\xi_i \xi_j] = \delta_{ij}$ due to orthonormality of the eigenfunctions. Then, \eqref{eq:proj_pca} can be equivalently written as 
    \[
    u^\textup{KL}_r = m_0 + \sum_{i=0}^r \sqrt{\lambda_i^{\textup{KL}}} \psi_i^{\textup{KL}} \xi_i. 
    \]
\end{remark}

Among all rank-$r$ orthogonal projectors in $\mathbb{U}$, denoted $\mathcal{Q}_r$, the projector $P^\textup{KL}_r$ yields the $L^2(\mu)$-optimal reconstruction of $u \sim \mu$, where the reconstruction error is equal to the sum of trailing eigenvalues of $\mathcal{C}$. We restate \cite[Prop. 22]{Luo2025} below. 

\begin{proposition}[$L^2(\mu)$-optimal approximation of the perturbation field.]
    \label{prop:kl}
    Let $u_{Q_r} = m_0 + Q_r (u - m_0)$ for $u \in \mathbb{U}$ and $Q_r \in \mathcal{Q}_r$. Then the approximation \eqref{eq:proj_pca} using $P^\textup{KL}_r \in \mathcal{Q}_r$ satisfies
    \[
    \label{eq:approx_pca}
    \mathbb{E}_{u \sim \mu} \Big[ \left\| u - u_r^\textup{KL} \right\|^2_\mathbb{U} \Big] = \underset{Q_r \in \mathcal{Q}_r}{\textup{min}} \mathbb{E}_{u \sim \mu} \Big[ \left\| u - u_{Q_r} \right\|^2_\mathbb{U} \Big] = \sum_{i=r+1}^\infty \lambda_i^{\textup{KL}} \ .
    \]
\end{proposition}

\subsubsection{Derivative-informed subspaces}
\label{sec:as}

Whereas the KL basis identifies directions of $\mathbb{U}$ along which the stochastic process has maximum variance, the derivative-informed (DI) basis identifies directions in $\mathcal{E} \subset \mathbb{U}$ along which a functional $F$ has maximum mean squared sensitivity. In this section, we review results on derivative-informed subspaces of infinite-dimensional Hilbert spaces, which were first formulated in \cite{Luo2025}.

The derivative-informed subspace is defined with respect to a functional of interest, denoted $F: \mathbb{U} \to \mathcal{Y}$. We first state our assumptions on $F$.

\begin{condition}[Conditions on $F$]
    \label{cond:f}
    The operator $F: \mathbb{U} \to \mathcal{Y}$ belongs to $C^1(\mathbb{U},\mathcal{Y})$ and its Fr\'echet derivative belongs to $L^2(\mu,\mathcal{L}(\mathbb{U},\mathcal{Y}))$. 
\end{condition}
Let $D_\mathcal{E}F(u): \mathcal{E} \to \mathcal{Y}$ be the Fr\'echet derivative of $F$ at $u$ restricted to the Cameron-Martin space and $D_\mathcal{E}F(u)^*: \mathcal{Y} \to \mathcal{E}$ its adjoint with respect to the $\mathcal{E}$-inner product. One can derive the derivative-informed basis from the spectral decomposition of a particular operator associated with the sensitivity of the functional $F$. 

\begin{definition}[Derivative-informed (DI) basis]
\label{def:dibasis}
Let $\mathcal{K}|_\mathcal{E}: \mathcal{E} \to \mathcal{E}$ be the \textit{sensitivity operator}
    \begin{equation}
    \label{eq:egop_e}
    \mathcal{K}|_\mathcal{E}w \coloneqq \mathbb{E}_{u \sim \mu}\big[ D_\mathcal{E}F(u)^*  D_\mathcal{E}F(u) (w) \big], \quad w \in \mathcal{E} \ .
    \end{equation}
The derivative-informed (DI) basis $\{\psi^{\textup{DI}}_i\}_{i\geq 0}$ is the set of eigenfunctions of $\mathcal{K}|_\mathcal{E}$, solving
    \begin{equation}
    \label{eq:eig_as}
    \mathcal{K}|_\mathcal{E} \psi^{\textup{DI}}_i=\lambda^{\textup{DI}}_i \psi^{\textup{DI}}_i \ .
    \end{equation}
\end{definition}

When $F$ satisfies \Cref{cond:f}, then $D_\mathcal{E}F(u)$ is Hilbert-Schmidt for $\mu$-almost surely all $u$ and $\mathcal{K}|_\mathcal{E}$ is self-adjoint, positive semidefinite, and trace-class. As a result, the eigenfunctions $\{\psi^{\textup{DI}}_i\}_{i \geq 0}$ form an orthonormal basis of $\mathcal{E}$ and the eigenvalues $\lambda^{\textup{DI}}_0 \geq \lambda^{\textup{DI}}_1 \geq \dots \geq 0$ are non-negative. The proof of \Cref{lem:traceclass} is in \Cref{app:sensitivitypf}. 

\begin{lemma}[Sensitivity operator is self-adjoint, positive semidefinite, and trace-class]
    \label{lem:traceclass}
    For $F$ which satisfies \Cref{cond:f}, the sensitivity operator $\mathcal{K}|_\mathcal{E}$ defined in \eqref{eq:egop_e} is a self-adjoint, positive semidefinite, trace-class operator on $\mathcal{E}$.
\end{lemma}

A further consequence of \Cref{cond:f} is that we can equivalently formulate \eqref{eq:eig_as} as a generalized eigenvalue problem. This formulation is leveraged when constructing discrete solutions of the derivative-informed basis in \Cref{sec:discretization}. The proof of \Cref{lem:gen_eig} is given in \Cref{app:adjointpf}. 

\begin{lemma}[Derivative-informed basis as the solution to a generalized eigenvalue problem] 
\label{lem:gen_eig}
For strictly positive definite covariance $\mathcal{C}$ and \Cref{cond:f}, $\mathcal{K}|_\mathcal{E} = \mathcal{C} \mathcal{K}$ where $\mathcal{K}:\mathbb{U} \to \mathbb{U}$ is given by
    \begin{equation}
        \label{eq:egop}
        \mathcal{K} h \coloneqq \mathbb{E}_{u \sim \mu}\big[ DF(u)^* DF(u) (h) \big], \quad h \in \mathbb{U} ,
    \end{equation}
and the derivative-informed basis in \Cref{def:dibasis} also solves the generalized eigenvalue problem defined by $(\mathcal{K}, \mathcal{C}^{-1})$, i.e.,
    \begin{equation}
        \label{eq:geneig}
        \mathcal{K} \psi^{\textup{DI}}_i=\lambda^{\textup{DI}}_i \mathcal{C}^{-1} \psi^{\textup{DI}}_i \ .
    \end{equation}
\end{lemma}
%
From the derivative-informed basis $\{\psi^{\textup{DI}}_r\}_{i \geq 0}$, a projector $P^{\textup{DI}}_r: \mathbb{U} \to \mathcal{E}$ for $r \leq \text{rank}(\mathcal{K}|_\mathcal{E})$ is then constructed as
\[
P^{\textup{DI}}_r = \sum_{i=0}^r \psi_i^{\textup{DI}} \otimes \mathcal{C}^{-1} \psi_i^{\textup{DI}} \ . 
\]
Note that $P^{\textup{DI}}_r$ is non-orthogonal with respect to the $\mathbb{U}$-inner product. However, under \Cref{cond:f}, the projector restricted to elements of the Cameron-Martin space, $P^{\textup{DI}}_r|_\mathcal{E} = \sum_{i=0}^r \psi^{\textup{DI}}_i \otimes_\mathcal{E} \psi^{\textup{DI}}_i$ where $\psi \otimes_\mathcal{E} \psi \coloneqq  \psi \la \psi, \cdot \ra_\mathcal{E}$, is an orthogonal projection in $\mathcal{E}$.

We generalize the result in \cite[Lemma 3.1]{Constantine2020} to the infinite-dimensional case in \Cref{lem:directderiv}, which shows that the eigenvalues of the derivative-informed basis have a variational interpretation: each eigenvalue $\lambda^{\textup{DI}}_i$ is the largest mean squared directional derivative of $F$ along a unit Cameron--Martin direction that is orthogonal to the preceding eigenfunctions, and the maximum is attained in the direction $\psi^{\textup{DI}}_i$. The proof of \Cref{lem:directderiv} is in \Cref{app:ddpf}. 

\begin{lemma}
    \label{lem:directderiv}
    Let $F$ be the observable defining $\mathcal{K}|_\mathcal{E}$, and suppose it satisfies \Cref{cond:f}. Let $\{(\lambda^{\textup{DI}}_i, \psi^{\textup{DI}}_i)\}_{i \geq 0}$ denote the eigenpairs of $\mathcal{K}|_\mathcal{E}$ solving \eqref{eq:eig_as}, ordered so that $\lambda^{\textup{DI}}_i \geq \lambda^{\textup{DI}}_{i+1}$. For each $i \geq 0$, define the set of admissible directions
    \[
    \mathcal{W}_i \coloneqq \bigl\{ w \in \mathcal{E} \;:\; \|w\|_\mathcal{E} = 1, \ \langle w, \psi^{\textup{DI}}_k \rangle_\mathcal{E} = 0 \ \text{for } k = 0, \dots, i-1 \bigr\}.
    \]
    Then, for every $i \geq 0$,
    \[
    \lambda^{\textup{DI}}_i
    = \max_{w \in \mathcal{W}_i} \, \mathbb{E}_{u \sim \mu}\bigl[ \| D_\mathcal{E} F(u)(w) \|^2_\mathcal{Y} \bigr]
    = \mathbb{E}_{u \sim \mu}\bigl[ \| D_\mathcal{E} F(u)(\psi^{\textup{DI}}_i) \|^2_\mathcal{Y} \bigr].
    \]
\end{lemma}


\subsubsection{Subspace approximations of functionals on the Hilbert space}

For a given rank-$r$ projector $P_r$ on the Hilbert space $\mathbb{U}$, \Cref{prop:optapprox} gives the analytical form for the optimal subspace approximation of a functional $F$ with respect to $P_r$, as provided in \cite[Sec. 2.1]{Zahm2020} and \cite[Sec. 3.3]{Luo2025}.

\begin{proposition}[Optimal rank-$r$ subspace approximation of $F$.]
\label{prop:optapprox}
    Let $\mathcal{H} \coloneqq \{ F: \mathbb{U} \to \mathcal{Y} \, | \, F \text{ meets \Cref{cond:f}} \}$. For any rank-$r$ projector $P_r$ in $\mathbb{U}$ for $r \in \mathbb{N}$, let $\mathcal{H}_{P_r} \subset \mathcal{H}$ denote the set of $\mathcal{G}_{P_r}$-measurable functions in $\mathcal{H}$, where $\mathcal{G}_{P_r}$ denotes the $\sigma$-algebra generated by $P_r$. For a Gaussian measure $\mu$ on $\mathbb{U}$, the $L^2(\mu, \mathcal{Y})$-optimal approximation of $F:\mathbb{U} \to \mathcal{Y}$ restricted to $\mathcal{H}_{P_r}$ is
    \[
        \mathbb{E}_\mu \big[ F| \mathcal{G}_{P_r} ] = \underset{F_r \in \mathcal{H}_{P_r}}{\textup{argmin}} \, \mathbb{E}_{u \sim \mu} \big[ \lVert F(u)-F_r(u) \rVert^2_\mathcal{Y}\big] \, .
    \]
\end{proposition}
This result follows from the Hilbert projection theorem and showing the conditional expectation is equal to the orthogonal projection of $F \in \mathcal{H}$ onto $\mathcal{H}_{P_r}$. Following \cite[Cor. 3.5.2]{Bogachev1998}, for $P_r$ independent of $(I-P_r)$ under the centered Gaussian measure $\mu_0 = \mathcal{N}(0,\mathcal{C})$, the conditional expectation with respect to $\mu_0$ can be expressed as
\begin{equation}
    \mathbb{E}_{\mu_0} \big[ F| \mathcal{G}_{P_r} \big](u) = \mathbb{E}_{z \sim \mu_0} \big[ F(P_ru + (I-P_r)z) \big] \, .
\end{equation}
Thus, the $L^2(\mu_0)$-optimal $\mathcal{G}_{P_r}$-measurable approximation is obtained by retaining the component of $u$ in the $r$-dimensional range of $P_r$ and averaging $F$ over the complementary component $(I-P_r)z$, with $z \sim \mu_0$. A simpler approximation of $F$ which is commonly used in practice is $F_r(u) = F(u_r)$,
where $u_r = m_0 + P_r(u-m_0)$ for $u \sim \mu = \mathcal{N}(m_0, \mathcal{C})$. This approximation, which neglects variation in the orthogonal complement of $P_r$, is equal to the optimum in \Cref{prop:optapprox} when $F$ is $\mu$-almost surely constant in the complementary directions. For instance, if $P_r$ is constructed from the first $r$ eigenfunctions of a spectral operator, then this condition is met when the trailing eigenvalues $(\lambda_{i})_{i \geq r+1}$ associated with the eliminated directions are all equal to zero. 

A special property of the derivative-informed basis is that it admits a certified error bound for the associated subspace approximation of $F$.

\begin{proposition}[Error bound on the rank-$r$ DI approximation of $F$.]
\label{prop:di}
Let $F^*_r$ denote the $L^2(\mu,\mathcal{Y})$-optimal approximation of $F$ in $\mathcal{H}_{P^{\textup{DI}}_r}$. Then
    \[
    \label{eq:approx_as}
    \mathbb{E}_{\mu} \Big[ \left\| F - F^*_r\right\|^2_\mathcal{Y} \Big] \leq \sum_{i=r+1}^\infty \lambda_i^{\textup{DI}}. 
    \]
\end{proposition}
\begin{proof}
    We note that the derivative-informed subspace amounts to a PCA of the Fr\'echet derivative of $F$. As shown in \cite[Prop. 26]{Luo2025}, we therefore have
\begin{equation}
    \label{eq:approx_as_der}
    \mathbb{E}_{\mu} \Big[ \left\| D_\mathcal{E} F (I - P_r^\textup{DI}) \right\|^2_{\textup{HS}(\mathcal{E}, \mathcal{Y})} \Big] = \underset{Q_r \in \mathcal{Q}_r}{\text{min}} \mathbb{E}_{\mu} \Big[ \left\| D_\mathcal{E} F (I - Q_r) \right\|^2_{\textup{HS}(\mathcal{E}, \mathcal{Y})} \Big] = \sum_{i=r+1}^\infty \lambda_i^{\textup{DI}} \ .
\end{equation}
It follows from the subspace Poincar\'e inequality \cite[Thm. 28]{Luo2025} that
    \begin{equation}
        \label{eq:poincare}
        \mathbb{E}_{\mu} \Big[ \left\| F - \mathbb{E}_\mu [ F|\mathcal{G}_{P^{\textup{DI}}_r}]\right\|^2_\mathcal{Y} \Big] \leq \mathbb{E}_{\mu} \Big[ \left\| D_\mathcal{E} F (I - P_r^\textup{DI}) \right\|^2_{\textup{HS}(\mathcal{E}, \mathcal{Y})} \Big] \ .
    \end{equation}
Combining \eqref{eq:approx_as_der} and \eqref{eq:poincare} returns the desired result.
\end{proof}

\section{Main results}
\label{sec:contribution}

\subsection{Dimension reduction of SDEs based on the sensitivity of path space observables}

This section discusses the infinite-dimensional formulation of derivative-informed subspaces of drift functions of SDEs based on the sensitivity of a chosen path space observable.  

\subsubsection{Hilbert space of drift perturbations}
\label{sec:hilbertdrift}

Throughout, we consider functions defined on the finite-dimensional state space $\mathcal{X} \subseteq \mathbb{R}^n$. In \Cref{lem:controlspace}, we show that boundedness of elements in $\mathbb{U}$ is sufficient to guarantee that any perturbation field $u \in \mathbb{U}$ admits a solution to the SDE in \eqref{eq:sde} which is unique in law. 


\begin{lemma}[Unique weak solution of the perturbed SDE]
    \label{lem:controlspace}
    Let $\mathbb{U}$ be a Hilbert space continuously embedded in $L^\infty(\mathcal{X})$. Then for any $u \in \mathbb{U}$, the SDE in \eqref{eq:sde} has a unique weak solution.
\end{lemma}

\begin{proof}
    Under the linear growth and Lipschitz continuity assumptions of \Cref{cond:driftdiff}, the SDE $\upd X^0_t = b(X^0_t) \upd t + \sigma (X^0_t) \upd W_t$ has a unique strong solution $X^0$, meaning it is unique in law. Since $u \in \mathbb{U}$ is bounded, the solution $X^u$ to the SDE in \eqref{eq:sde} exists and is also unique in law, by \cite[Ch.IX, Thm.1.11]{RevuzYor1999}.
\end{proof}

A canonical choice of Hilbert space which is continuously embedded in $L^\infty(\mathcal{X})$ is the Sobolev space $H^m(\mathcal{X})$ with $m > \frac{n}{2}$. Recall that $H^m(\mathcal{X})$ is the Sobolev class of functions on $\mathcal{X}$ whose $k$th-order weak partial derivatives up to order $m \geq 0$ are square integrable, i.e.,
\[
H^m(\mathcal{X}) \coloneqq \{u: \mathcal{X} \to \mathcal{X} \ | \ \partial_k u \in L^2(\mathcal{X}) \text{ for } k =0,1,...,m\} \ .
\]
From the Sobolev embedding theorem, we have that for $m > \frac{n}{2}$ and $0 < \eta < m - \frac{n}{2}$, $H^m(\mathcal{X}) \subseteq C^\eta(\mathcal{X})$ \cite[App. A.1]{Nickl2023}, where $C^\eta(\mathcal{X})$ is a H\"older space with the norm $\|u\|_{C^\eta}=\sum_{|k|\leq \eta} \| \partial_k u \|_\infty$. It follows from \cite[Eq. (A.4)]{Nickl2023} that any $u \in H^m(\mathcal{X})$ has finite $L^\infty(\mathcal{X})$ norm, since
\[
\| u \|_\infty \lesssim \| u \|_{C^\eta} \lesssim \|u\|_{H^m} \, .
\]
    

We can define Gaussian measures on $H^m(\mathcal{X})$ as follows. 

\begin{fact}
    \label{fact:measure}
    Let $\mathbb{U} = H^m(\mathcal{X})$ with $m \geq \frac{n}{2}$. For a Laplacian-like operator $\mathcal{A}$ densely defined on $\mathbb{U}$ that satisfies \cite[Assump.~2.9]{Stuart2010}, $\mu = \mathcal{N}(0,\mathcal{A}^{-\nu})$ with $\nu > \frac{n}{2}$ is a Gaussian measure on $\mathbb{U}$.
\end{fact}

This fact follows from \cite[Lemma 6.27]{Stuart2010}, which states that under these assumptions, $u \sim \mu$ is in the Hilbert scale space $\mathbb{U}^s$ \eqref{eq:hilbertscale} almost surely for $s \in [0, \nu - \frac{n}{2})$. For $\nu > \frac{n}{2}$, we see that $u \in \mathbb{U}^0 = \mathbb{U}$ for $\mu$-almost every $u$. Since $\nu$ controls the smoothness of draws $u$ from $\mu$, the condition $\nu > \frac{n}{2}$ ensures that the perturbation fields have sufficient regularity to belong to $H^m(\mathcal{X})$ almost surely.

A covariance operator which meets these assumptions is the regularized Laplacian
\begin{equation}
    \label{eq:laplacian}
    \mathcal{A} = -\gamma \Delta + \delta I \ ,
\end{equation}
where the hyperparameters $\gamma > 0 $ and $\delta > 0$ determine the correlation length $\rho$ and pointwise variance $\zeta$ of Gaussian processes $u$ whose law is $\mu = \mathcal{N}(0,\mathcal{A}^{-\nu})$. In particular, $\rho$ is proportional to $\sqrt{\gamma / \delta}$ and $\zeta$ is proportional to $\delta^{-\nu} \rho^{-n}$ \cite{Villa2021}, with the exact expressions derived in \cite{Lindgren2011}.

\subsubsection{Fr\'echet derivative of path space observables}
\label{sec:fd_pathobs}

Recall that for drift perturbations $u \in \mathbb{U}$ where $\mathbb{U}$ is a Hilbert space with properties specified in \Cref{lem:controlspace}, the perturbed SDE has a solution with law $\mathbb{P}^u$ and the path space observable $F(u) \coloneqq \mathbb{E}^u[\varphi_\tau]$, as introduced in \eqref{eq:observable}, is well-defined for $\varphi_\tau \in L^1(\mathbb{P}^u)$. In \Cref{lem:frechet}, we specify additional assumptions which are sufficient to show that the $n$th-order Fr\'echet derivative $D^nF(u)$ for $n \in \mathbb{N}$ exists and has an analytical form.

\begin{lemma}[Fr\'echet derivative of the path space observable]
    \label{lem:frechet}
    Assume \Cref{cond:driftdiff} holds and let $\mathbb{U}$ be specified as in \Cref{lem:controlspace}. For $u \in \mathbb{U}$, let $F(u)$ be the path space observable \eqref{eq:observable} where $\varphi_\tau$ is in $L^2(\mathbb{P}^u, \mathcal{Y})$ and satisfies \Cref{cond:stoptime}. Then the first-order Fr\'echet derivative of $F$ at $u \in \mathbb{U}$ in the direction $v \in \mathbb{U}$ is given by
    \begin{equation}
        \label{eq:frechet}
        D^nF(u)(v) \coloneqq \mathbb{E}^u \big[ \varphi_\tau M^{v}_\tau \big] \ ,
    \end{equation}
    where for any stochastic process $Y = (Y_t)_{t \in [0,\tau]}$ and $v \in \mathbb{U}$, $M^v_\tau$ is the local martingale
    \begin{equation}
        \label{eq:martingale}
        M^v_\tau(Y) = \int_0^\tau v(Y_s)^\top \upd W_s \ .
    \end{equation}
\end{lemma}

\begin{proof}
    The condition in \Cref{lem:controlspace} ensures that $u \in \mathbb{U}$ is bounded, which is a requirement specified for the function space in \cite[Eq. (22)]{Lie2021}. The semi-ellipticity condition in \Cref{cond:driftdiff} corresponds to Assumption 3.1 of \cite{Lie2021}. Since the stopping time is bounded, it satisfies Assumption 3.2 of \cite{Lie2021}. The result then matches \cite[Lemma 3.6]{Lie2021}. 
\end{proof}
%

\subsubsection{Derivative-informed subspaces for path space observables}
\label{sec:as_pathobs}

In this section, we apply the framework of derivative-informed subspaces from \Cref{sec:as} to path space observables of SDEs of the form in \eqref{eq:observable}. For the derivative-informed subspace to be well-defined, a sufficient condition is to require additional properties of the path functional $\varphi_\tau$ with respect to a given Gaussian measure on $\mathbb{U}$, detailed in \Cref{cond:pathfunc}. 

\begin{condition}[Conditions on the path functional]
    \label{cond:pathfunc}
    Let $\mu$ be a Gaussian measure on $\mathbb{U}$. The path space observable $F$ in \eqref{eq:observable} is defined by a path functional $\varphi_\tau$ which
    
    \begin{enumerate}[label=\roman*.] 
    \item is $\mathbb{P}^u$-almost surely nonconstant for almost surely all $u \sim \mu$,
    \item satisfies \Cref{cond:stoptime},
    \item satisfies $\Big( \int_\mathbb{U} \lVert \varphi_\tau \rVert^4_{L^2(\mathbb{P}^u, \mathcal{Y}) } \mu (\upd u) \Big)^{\frac{1}{2}} < \infty$.
    \end{enumerate}
\end{condition}
As an example, \Cref{cond:pathfunc} is met for bounded mean first hitting time observables $\varphi_\tau(X^u) = \tau_\infty(X^u) \wedge T$ where $\tau$ is $\mathbb{P}^u$-almost surely nonconstant for $u \sim \mu$. 

We obtain \Cref{prop:condf} as our key result, implying that under the appropriate conditions, the path space observable of the perturbed SDE is a suitable functional for forming a derivative-informed subspace. The proof of \Cref{prop:condf} is in \Cref{app:condfpf}.

\begin{proposition}[Path-space observable admits a derivative-informed subspace]
    \label{prop:condf}
    Assume \Cref{cond:driftdiff} holds and let $\mathbb{U}$ be specified as in \Cref{lem:controlspace}. For $u \in \mathbb{U}$, let $F(u)$ be the path space observable defined in \eqref{eq:observable} where $\varphi_\tau$ satisfies \Cref{cond:pathfunc}. Then, $F(u)$ and $DF(u)$ satisfy \Cref{cond:f} for a given Gaussian measure $\mu$ on $\mathbb{U}$ and the path space sensitivity operator $\mathcal{K}|_{\mathcal{E}} \coloneqq \mathbb{E}_\mu[(D_\mathcal{E} F)^*D_\mathcal{E} F]$ on $\mathcal{E} \subset \mathbb{U}$ is self-adjoint, positive semidefinite, and trace-class. 
\end{proposition}
Hence, the \textit{path-space sensitivity operator} $\mathcal{K}|_\mathcal{E}$ in \Cref{prop:condf} inherits desirable properties established in \Cref{sec:as} for general derivative-informed subspaces; in particular, that its eigenfunctions are ordered in terms of the mean squared sensitivity of the observable, by \Cref{lem:directderiv}, and they admit low-rank approximations of the observable with the error guarantees of \Cref{prop:di}. 

\subsection{Discretization of the dimension reduction problem}
\label{sec:discretization}

We now discuss the numerical discretization of the derivative-informed dimension reduction approach. In practice, elements of the infinite-dimensional space of perturbation functions, $\mathbb{U}$, are approximated by their projection onto a finite-dimensional subspace of $\mathbb{U}$. We focus our attention on finite element subspaces which provide consistent approximations of elements of $\mathbb{U}$.

\subsubsection{Finite element basis}
\label{sec:fe}

Consider the space of finite element (FE) discretizations $\mathbb{U}_h \subset H^1(\mathcal{X})$ for $n \in \{1,2\}$, defined by a mesh size $h > 0$, a set of nodal points $\{x_j\}_{j=1}^J$, and a nodal basis of continuous Lagrange functions $\{\phi_j\}_{j=1}^J$ satisfying $\phi_j(x_i) = \delta_{ij}$ for $i,j \in \{1,...,J\}$. For every $u \in H^1(\mathcal{X})$, the FE approximation
\[
u_h(\bm{\alpha}) = \sum_{j=1}^J \alpha_j \phi_j
\]
is fully characterized by its nodal coefficient vector $\bm{\alpha} = (\alpha_1,...,\alpha_J)$. The FE approximation has the representation $u_h = P_hu$, where $P_h: H^1(\mathcal{X}) \to \mathbb{U}_h$ is a bounded linear projection onto $\mathbb{U}_h$. It is a consistent discretization in the sense that under regular mesh refinement, $u_h \to u$ in $H^1(\mathcal{X})$ as $h \to 0$; see \cite{BrennerScott2008}.

The FE subspace of $L^2(\mathcal{X})$ specified in \cite{BuiThanh2013,Villa2021} is also an FE subspace of $H^1(\mathcal{X})$, as the Lagrange basis functions are bounded and their weak derivatives are in $L^2(\mathcal{X})$. The infinite-dimensional $L^2$-inner product between two functions is approximated by the $\bm{M}$-weighted Euclidean inner product between their nodal coefficient vectors, where the FE mass matrix $\bm{M} \in \mathbb{R}^{J \times J}$ is
\begin{equation}
    \label{eq:massmatrix}
    \bm{M}_{ij} = \int_{\mathcal{X}} \phi_i(x)\phi_j(x) \upd x , \quad i,j \in \{1,...,J\} \ .
\end{equation}
Then, one has that $\la u_1,u_2 \ra_{\mathbb{U}} \approx \la u_{1h},u_{2h} \ra_{L^2} = \la \bm{\alpha}_1, \bm{\alpha}_2 \ra_{\bm{M}} = \bm{\alpha}_1^\top \bm{M} \bm{\alpha}_2$. We let $\mathbb{R}^J_{\bm{M}}$ denote $\mathbb{R}^J$ equipped with the $\bm{M}$-weighted inner product.

\subsubsection{Discrete representation of the derivative-informed basis}
\label{sec:galerkin}

In the following, we derive matrix representations of infinite-dimensional operators on $\mathbb{U}$ relevant to the dimension reduction problem using Galerkin projections onto the FE subspace. 

\medskip 

\noindent \textit{Discrete representation of $\mathcal{C}$.}
The covariance of the Gaussian measure $\mu$ is taken to be the fractional Laplacian operator $\mathcal{A}^{-\nu}$, where $\mathcal{A}$ is defined in \eqref{eq:laplacian}. The matrix representation $\bm{A}:\mathbb{R}^J_{\bm{M}} \to \mathbb{R}^J_{\bm{M}}$ of $\mathcal{A}$ satisfies
    \begin{equation}
        \label{eq:aoperator}
        \int_{\mathcal{X}} \phi_i(x) \mathcal{A} \phi_j(x) \upd x = \la \bm{e}_i, \bm{A} \bm{e}_j \ra_{\bm{M}} \ ,
    \end{equation}
where $\bm{e}_i$ is the nodal coefficient vector corresponding to $\phi_i$. For $i,j \in \{1,...,J\}$, the stiffness matrix $\bm{S}$ is defined as
    \begin{equation}
        \label{eq:stiffmatrix}
        \begin{aligned}
        (\bm{S})_{ij} & = \int_{\mathcal{X}} \phi_i(x) \mathcal{A} \phi_j(x) \upd x = \int_{\mathcal{X}} \phi_i(x) \big( -\gamma \Delta + \delta I \big) \phi_j(x) \upd x \\
        & = \gamma \int_{\mathcal{X}} \nabla \phi_i(x) \cdot \nabla \phi_j(x) \upd x + \delta \int_{\mathcal{X}} \phi_i(x) \phi_j(x) \upd x \ .
        \end{aligned}
    \end{equation}
The stiffness matrix $\bm{S}$ corresponds to the bilinear form and the mass matrix $\bm{M}$ represents the linear functional in the weak formulation of $\mathcal{A}$. Therefore, one can show that the discrete operator $\bm{A} = \bm{M}^{-1} \bm{S}$ is the matrix representation of $\mathcal{A}$ with respect to the $\bm{M}$-weighted inner product. It follows that the matrix representation of the covariance $\mathcal{C} = \mathcal{A}^{-\nu}$ is $\bm{C} = \bm{A}^{-\nu}$.  

\medskip

\noindent \textit{Discrete representation of $\mathcal{K}$.}
For $u \in \mathbb{U}$, we approximate $F(u)$ by $F(u_h(\bm{\alpha}))$, where $u_h = P_hu$. 
Using \Cref{lem:frechet}, the partial derivative of $F(u_h(\bm{\alpha}))$ with respect to the $i$th nodal coefficient is 
\[
    \partial_{\alpha_i}F(u_h(\bm{\alpha})) = DF(u_h)(\partial_{\alpha_i}u_h) = DF(u_h)(\phi_i) =\mathbb{E}^{u_h}[\varphi_\tau M^{\phi_i}_\tau] \, ,
\]
since $\partial_{\alpha_i} u_h = \phi_i$. We let $G(u) = \nabla_{\bm{\alpha}} F(u_h(\bm{\alpha}))$ denote the $J$-dimensional gradient of the path space observable defined on the FE subspace. The Galerkin projection of the sensitivity operator $\mathcal{K}$ \eqref{eq:egop} onto the FE subspace yields the matrix representation $\bm{M}^{-1} \bm{K}$ where $\bm{K}$ is given by
\begin{equation}
    \label{eq:egop_discrete}
    \bm{K} = \mathbb{E}_{u \sim \mu}[G(u)G(u)^\top] \ .
\end{equation}
In practice, the sensitivity matrix is computed using Monte Carlo estimation for both the inner and outer expectations. The gradient, given by the inner expectation with respect to the path measure $\mathbb{P}^{u_h}$, is approximated by the $l$-sample Monte Carlo estimator
\begin{equation}
    \label{eq:grad_mc}
     G(u) \approx \hat{G}_l(u) \coloneqq \frac{1}{l} \sum_{i=1}^l \varphi_\tau(X^{(i)})M^\Phi_\tau (X^{(i)}), \quad X^{(i)} \sim \mathbb{P}^{u_h},
\end{equation}
where $\Phi=(\phi_1,...,\phi_J)$ corresponds to the Jacobian matrix of $u_h$ with respect to $\bm{\alpha}$. The sensitivity matrix, given by the outer expectation with respect to the Gaussian measure $\mu = \mathcal{N}(0,\mathcal{A}^{-\nu})$, is approximated by the $m$-sample Monte Carlo estimator 
\begin{equation}
\label{eq:egop_mc}
 \bm{K} \approx \hat{\bm{K}}_{l,m} \coloneqq \frac{1}{m} \sum_{j=1}^m \hat{G}_l(u^{(j)}) \hat{G}_l(u^{(j)})^\top, \quad u^{(j)} \sim \mu.
\end{equation}

\noindent \textit{Discrete representation of the eigenvalue problem.}
The finite-dimensional approximation of the eigenvalue problem \eqref{eq:geneig} is the generalized matrix eigenvalue problem
\begin{equation}
    \label{eq:geneig_discrete}
    \hat{\bm{K}}_{l,m} \bm{\psi}^{(i)} = \lambda^{(i)} \bm{M} \bm{C}^{-1} \bm{\psi}^{(i)},
\end{equation}
where $\bm{\psi}^{(i)} \in \mathbb{R}^J_{\bm{M}}$ denotes the nodal coefficient vector corresponding to the eigenfunction $\psi_{ih} \in \mathbb{U}_h$ for $i = 1,..., J$. The orthogonality condition $\la \psi_i, \psi_j \ra_\mathcal{E} = \delta_{ij} \ \forall \ i,j \in \{1,...,J\}$ is equivalent to $ \Psi^\top \bm{M} \bm{C}^{-1} \Psi = \mathbb{I}$ in the nodal coefficient space, since $\la \cdot, \cdot \ra_\mathcal{E} = \la \mathcal{C}^{-\half} \cdot, \mathcal{C}^{-\half} \cdot \ra_{\mathbb{U}} = \la \cdot, \mathcal{C}^{-1} \cdot \ra_{\mathbb{U}}$ for nondegenerate $\mathcal{C}$ and 
\[
\la \psi_i, \psi_j \ra_\mathcal{E} = \la \psi_i, \mathcal{C}^{-1} \psi_j \ra_{\mathbb{U}} \approx \la  \psi_{ih}, \mathcal{C}^{-1} \psi_{jh} \ra_{L^2} = \la \bm{\psi}^{(i)}, \bm{C}^{-1} \bm{\psi}^{(j)} \ra_{\bm{M}} \, .
\]

\medskip

\noindent \textit{Discrete representation of the projector.}
The projection operator $P^\text{DI}_r|_\mathcal{E}$ from \Cref{sec:as} can be represented in the FE subspace as follows. For $u \in \mathcal{E}$ with FE approximation $u_h \in \mathbb{U}_h$ 
and nodal coefficients $\bm{\alpha} \in \mathbb{R}^J_{\bm{M}}$,
\begin{equation}
    \begin{aligned}
        P^\text{DI}_r|_\mathcal{E} u 
        & = \sum_{i=1}^r \psi_{i} \langle \psi_{i}, u \rangle_\mathcal{E} \approx \sum_{i=1}^r \psi_{ih} \langle \psi_{ih}, \mathcal{C}^{-1} u_h \rangle_{L^2} \\
        & = \sum_{i=1}^r \bm{\psi}^{(i)} \langle \bm{\psi}^{(i)}, \bm{C}^{-1}\bm{\alpha} \rangle_{\bm{M}} = \sum_{i=1}^r \bm{\psi}^{(i)} (\bm{\psi}^{(i)})^\top \bm{M} \bm{C}^{-1} \bm{\alpha} = \Psi_r \Psi_r^\top \bm{M} \bm{C}^{-1} \bm{\alpha}.
    \end{aligned}
\end{equation}
Therefore, the matrix representation of the projector is 
$\bm{P}^\text{DI}_r|_\mathcal{E} = \Psi_r \Psi_r^\top \bm{M} \bm{C}^{-1}$.

\subsection{Control variate estimation of the derivative-informed subspace}
\label{sec:controlvariate}

Numerical estimation of the derivative-informed basis is subject to variability arising from the Monte Carlo estimation in \eqref{eq:egop_mc}. We show that a control variate formulation of the Fr\'echet derivative of the path space observable leads to reduced variance estimation of the sensitivity matrix and, subsequently, reduced error in the rank-$r$ derivative-informed subspace. 

In \cite{Lie2021}, it is proposed that the local martingale in \eqref{eq:martingale} may be used to construct a control variate of the Fr\'echet derivative in \Cref{lem:frechet}. Since the martingale has the property that for any $v \in \mathbb{U}$, $\mathbb{E}[M^v_\tau]=0$ \cite[(3.19)]{Pavliotis2014}, it follows that for any $u \in \mathbb{U}$, $\mathbb{E}^u[M^v_\tau]=0$. Thus, for any $\beta \in \mathbb{R}$, the quantity $\mathbb{E}^u[\varphi_\tau M^v_\tau] + \beta \mathbb{E}^u[M^v_\tau]$ is an unbiased estimator of $DF(u)(v) = \mathbb{E}^u[\varphi_\tau M^v_\tau]$. Since $M^v_\tau$ is already computed in the course of calculating $\varphi_\tau M^v_\tau$ and is correlated by construction, it is a natural choice of control variate that reduces variance of the estimator at no additional cost. Using the finite element discretization scheme from \Cref{sec:discretization}, a control variate estimator of the gradient of the observable, $\hat{G}_l^{\textup{CV}}:\mathbb{U} \to \mathbb{R}^J$, is defined as
\begin{equation}
    \label{eq:grad_cv}
    \hat{G}_l^{\textup{CV}}(u) \coloneqq \frac{1}{l} \sum_{i=1}^l \big( \varphi_\tau(X^{(i)})M^\Phi_\tau (X^{(i)}) \big)  + \beta \frac{1}{l} \sum_{i=1}^l M^\Phi_\tau (X^{(i)}), \quad X^{(i)} \sim \mathbb{P}^{u_h}.
\end{equation}
By expanding $\text{Var}_{\mathbb{P}^u}(\varphi_\tau M^v_\tau + \beta M^v_\tau)$, one can show that the optimal choice of $\beta$ for a minimum variance estimator is given by
\begin{equation}
    \label{eq:optimalbeta}
    \beta^* = \frac{\text{Cov}_{\mathbb{P}^u}(\varphi_\tau M^v_\tau, M^v_\tau)}{\text{Var}_{\mathbb{P}^u}(M^v_\tau)} \, , \quad  \text{Var}_{\mathbb{P}^u}(M^v_t)=\mathbb{E}^u \Big[ \int_0^\tau |v(X^u_s)|^2 \upd s \Big] ,
\end{equation}
where the variance of $M^v_\tau$ follows from It\^o isometry \cite[(3.21)]{Pavliotis2014}. A control variate-based estimator of the sensitivity matrix in \eqref{eq:egop} is then given by
\begin{equation}
    \label{eq:egop_cv}
    \hat{\bm{K}}^\textup{CV}_{l,m} \coloneqq \frac{1}{m} \sum_{j=1}^m \hat{G}^\textup{CV}_l(u^{(j)}) \hat{G}^\textup{CV}_l(u^{(j)})^\top, \quad u^{(j)} \sim \mu.
\end{equation}
While matrix-valued control variates for the sensitivity matrix in the outer expectation can be constructed (see \cite{Maurais2023}), they require additional samples of $u \sim \mu$ and hence additional numerical simulations of the SDE to obtain a useful finite-sample estimate of the control variate mean. Since the simulation of SDE paths dominates the cost of computing the derivative-informed subspace, such estimators are less practical compared to \eqref{eq:egop_cv}.







\section{Numerical studies}
\label{sec:numerics}

In the following numerical examples, we demonstrate the computation of derivative-informed (DI) subspaces of path-space observables of SDEs using FE discretizations of the infinite-dimensional problem. For this task we use hIPPYlib \cite{Villa2021}, a Python library for infinite-dimensional Bayesian inverse problems that is built upon the finite element infrastructure of FEniCS \cite{Alnaes2015}. In particular, FEniCS is used to define the FE mesh on which the covariance, sensitivity, and projection operators are discretized via Galerkin projections, while hIPPYlib provides code for sampling from Gaussian measures on FE subspaces of Hilbert spaces.

We study perturbed SDEs of the form 
    \begin{equation}
        \label{eq:controllangevin}
        \upd X^u_t = (-\nabla V + u)(X^u_t) \, \upd t + \upd W_t, \qquad X_0 = x \ ,
    \end{equation}
where $u \sim \mu = \mathcal{N}(0,\mathcal{A}^{-\nu})$ and $\mathcal{A}$ is as defined in \eqref{eq:laplacian}. For SDEs defined on $\mathbb{R}$ or $\mathbb{R}^2$, we take $\nu=2$ to satisfy the condition in \Cref{fact:measure} and also to obtain convenient forms of the square root operators as  $\mathcal{C}^{\half}=\mathcal{A}^{-1}$ and $\mathcal{C}^{-\half}=\mathcal{A}$. 

In both examples, we compare the approximation of the path-space observable using the DI subspace, derived from the operator pencil $(\mathcal{K},\mathcal{C}^{-1})$ \eqref{eq:geneig}, to that of the KL expansion, derived from $\mathcal{C}$ \eqref{eq:eig_pca}. The subspace approximation of $F$ is taken to be $F_r(u) = F(P_ru)$ for $u$ and projector $P_r$ on $\mathbb{U}$. We assess the quality of the subspace approximation using the mean squared error (MSE) in the observable and the Kolmogorov-Smirnov statistic between the pushforward distributions $F_\sharp \mu$ and $(F_r)_\sharp \mu$, which are defined in \Cref{app:errormetrics}.

\subsection{Double well potential}
\label{sec:num_dw}

Consider the asymmetric double well potential on $\mathbb{R}$, 
    \begin{equation}
        \label{eq:dw}
        V(x) = x^4 - 2x^2 + \frac{1}{2} \ \Big( \frac{1}{3} x^3 + x^2 + x \Big) \, .
    \end{equation}
We define the target path-space observable to be the mean first hitting time (MFHT) given a fixed initial condition and hitting domain $D = (1,\infty)$,
\begin{equation}
    \label{eq:bdmfht}
    F(u) = \mathbb{E}^u[\tau], \quad \tau = \inf\{t \geq 0|X_t \in D, X_0=-1\} \wedge T \ .
\end{equation}
The maximum time $T=60$ is chosen to be sufficiently large relative to the mean value of $F_\sharp\mu$. The potential and the hitting domain are illustrated in the top right panel of \Cref{fig:dw}.
    \begin{figure}[htbp]
        \centering
        \includegraphics[width=0.9\textwidth]{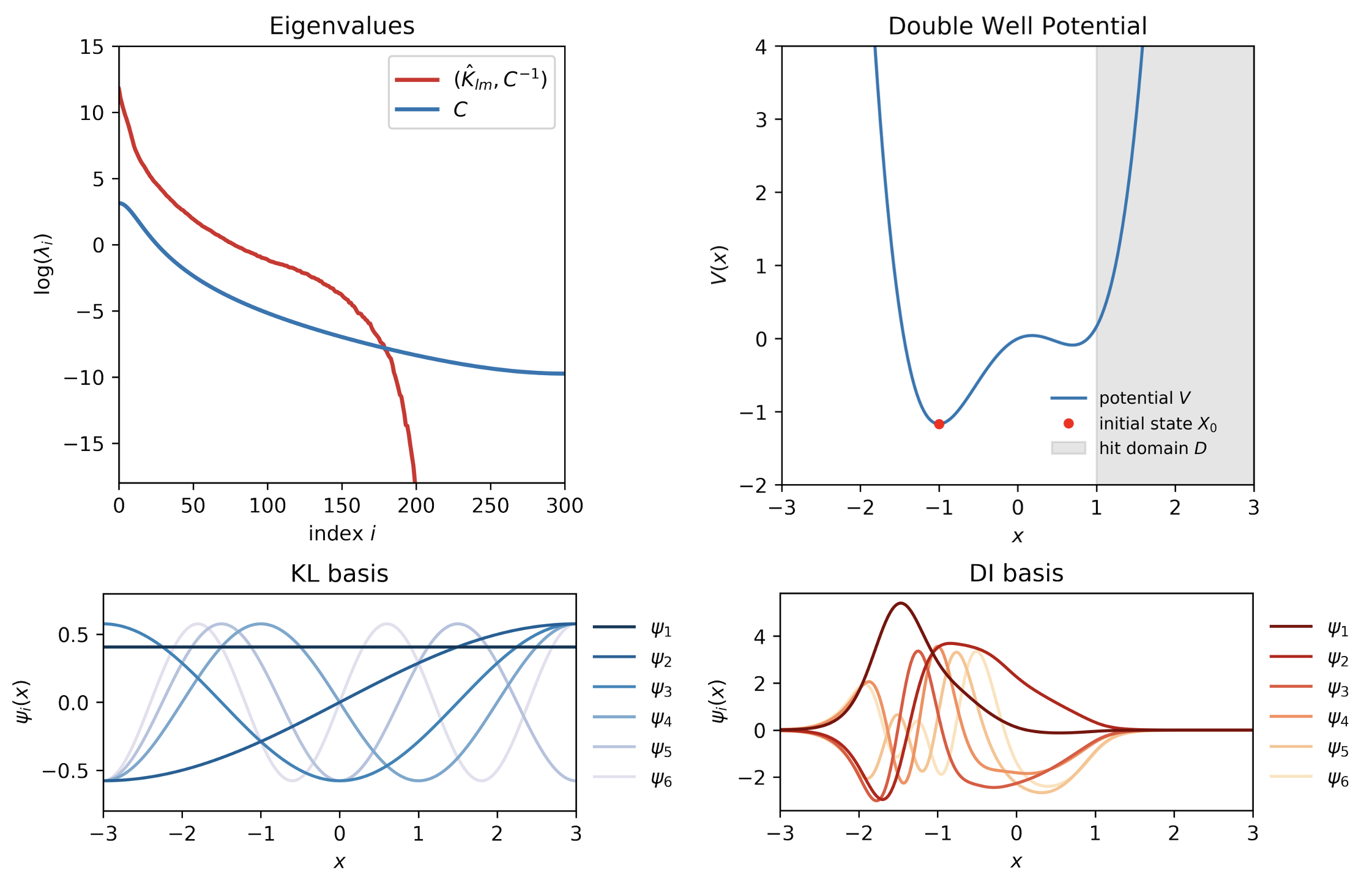}
        \caption{Top right: Asymmetric double well potential of \eqref{eq:dw}. The path-space observable is the mean first time within $[0,T]$ that paths starting at $X_0=-1$ hit the domain $D = (1,\infty)$. Top left: Eigenvalues of $(\bm{K},\bm{C}^{-1})$, corresponding to the DI basis, and of $\bm{C}$, corresponding to the KL basis. Bottom left: First six eigenfunctions of $\bm{C}$. Bottom right: First six eigenfunctions of $(\bm{K},\bm{C}^{-1})$.}
        \label{fig:dw}
    \end{figure}
    
We define a finite element mesh on the spatial domain $[-3,3]$ with homogeneous Neumann boundary conditions, as the dynamics of the SDE rarely extend beyond this domain. The mesh is constructed with $d=300$ elements, corresponding to $h = 0.02$. The parameters $\gamma, \delta$ of the Laplacian operator defining the measure $\mu$ are computed such that draws from $\mu$ are Gaussian processes with variance $\zeta=40$ and correlation $\rho=0.5$. The Monte Carlo estimator of the sensitivity operator in \eqref{eq:egop_mc} is computed using $m=10^4$ independent drift perturbations $u^{(1)},...,u^{(m)} \sim \mu$. For each $u^{(j)}$, we evaluate the path-space observable \eqref{eq:bdmfht} and its Fr\'echet derivative using $l=10^3$ sample paths associated with the path measure $\mathbb{P}^{u^{(j)}}$. To control variance of the estimator and facilitate the comparison of sample-based estimators, we use the same set of $l$ Brownian motion realizations to simulate paths of the perturbed SDE for each draw of $u$; this strategy ensures that differences in estimators arise only from variation in $u$ and not from the Brownian motion realization. 

We compare the spectral decomposition of $\bm{C}$ for deriving the KL basis and of $(\hat{\bm{K}}_{lm},\bm{C}^{-1})$ for deriving the DI basis. In the top left panel of \Cref{fig:dw}, we see that the eigenvalues of the DI basis exhibit a much sharper decay leading up to $i=200$ compared to those of the KL basis. Following \Cref{lem:directderiv}, this rapid spectral decay implies that the sensitivity of the observable to drift perturbations is captured within a low-dimensional subspace of the perturbations.

The bottom panels of \Cref{fig:dw} display the first six eigenfunctions of the KL and DI bases. In the one-dimensional setting considered here, the Laplacian-based covariance operator $\mathcal{C} = \mathcal{A}^{-2}$ defined on $L^2(\mathcal{X})$ with $\mathcal{X} \subseteq \mathbb{R}$ admits closed-form eigenvalues and eigenfunctions, which we provide in \Cref{app:eigfuncs}. The numerically computed KL eigenfunctions agree with their closed-form expressions up to sign and discretization error. These eigenfunctions depend only on the covariance operator $\mathcal{C}$ and thus remain agnostic to the path-space observable. In contrast, the leading DI eigenfunctions exhibit a distinct spatial pattern where variation is concentrated within the interval from $x=-3$ to the hitting domain at $x=1$ and trends toward zero at the boundaries. The leading eigenfunctions are observable-informed in the sense that their shape reveals features of the drift function which more strongly influence the observable, such as the geometry of the potential about the initial condition. 
    \begin{figure}[htbp]
    \centering
    \includegraphics[width=0.9\textwidth]{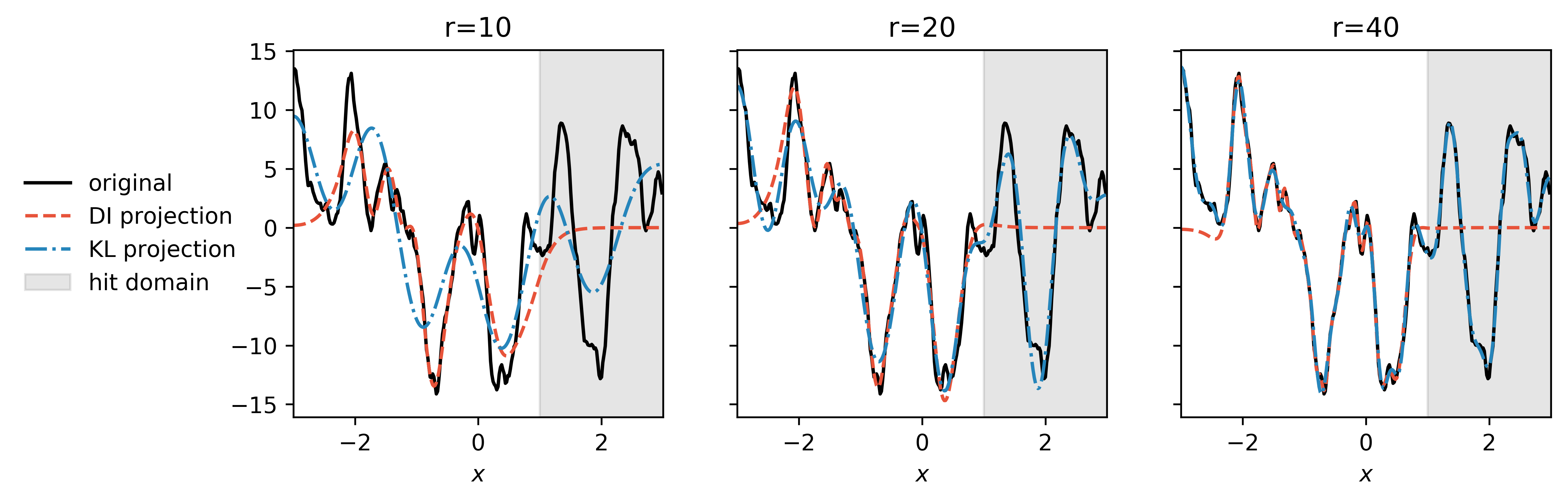}
    \caption{Projection of a sample drift perturbation onto $r$-dimensional subspaces of the KL basis and DI basis, for $r=[10, 20, 40]$ (left to right). }
    \label{fig:dw_proj}
    \end{figure}
    
To gain further insight, we illustrate the projection of a sample drift perturbation onto $r$-dimensional subspaces defined by the KL and DI bases in \Cref{fig:dw_proj}. By \Cref{prop:kl}, the KL subspace minimizes the $L^2(\mu)$ error in the rank-$r$ projection of the drift perturbation. This optimality is global with respect to $\mu$ and does not control error locally in the spatial domain, so the KL projection of an individual perturbation can exhibit local discrepancies, particularly for small $r$. The DI projection, on the other hand, has smaller local error with respect to the original perturbation field approximately on the interval $(-2.2,1)$, then tapers to zero near the boundaries and for $x>1$. Since the DI subspace is informed by the sensitivity of the observable to the perturbation, the DI projection neglects variation in the perturbation field in regions that the SDE trajectories rarely visit or that lie beyond the termination point.

After identifying the subspaces, we compute the KL and DI approximations of the mean first hitting time (MFHT) observable by sampling a perturbation $u \sim \mu$, applying the projection $u_r = P_r u$, then simulating the $u_r$-perturbed SDE with the fixed Brownian motion realization to calculate $F_r(u) = F(P_r u)$. The approximation error metrics are computed over $m=5 \times 10^3$ samples of the drift perturbation $u \sim \mu$ projected onto the $r$-dimensional subspaces for $r=[10, 20, 40, 80, 120, 160, 200, 240, 280, 300]$. The leftmost panel of \Cref{fig:dw_mfht} shows that the pushforward distribution $F_\sharp \mu$ computed with the full-dimensional perturbations, corresponding to $d=300$, is strongly bimodal, with peaks near $F = 0$ and the final time $F = T$. The second peak corresponds to perturbation samples for which SDE trajectory on average did not reach the hitting domain within the time frame $[0,T]$. For the same subspace rank $r$, the distribution of the DI approximations of the MFHT more accurately captures the modes of the distribution compared to the KL approximations. We then compute the $L^2(\mu)$ error in the observable and Kolmogorov-Smirnov (KS) statistic between $F_\sharp \mu$ and $(F_r)_\sharp \mu$.  The center and rightmost panels of \Cref{fig:dw_mfht} show that for every truncation order $r$ leading up to $d$, both error metrics using the DI approximations of the observable are lower compared to those with the KL approximations.
    \begin{figure}[htbp]
    \centering
    \includegraphics[width=\textwidth]{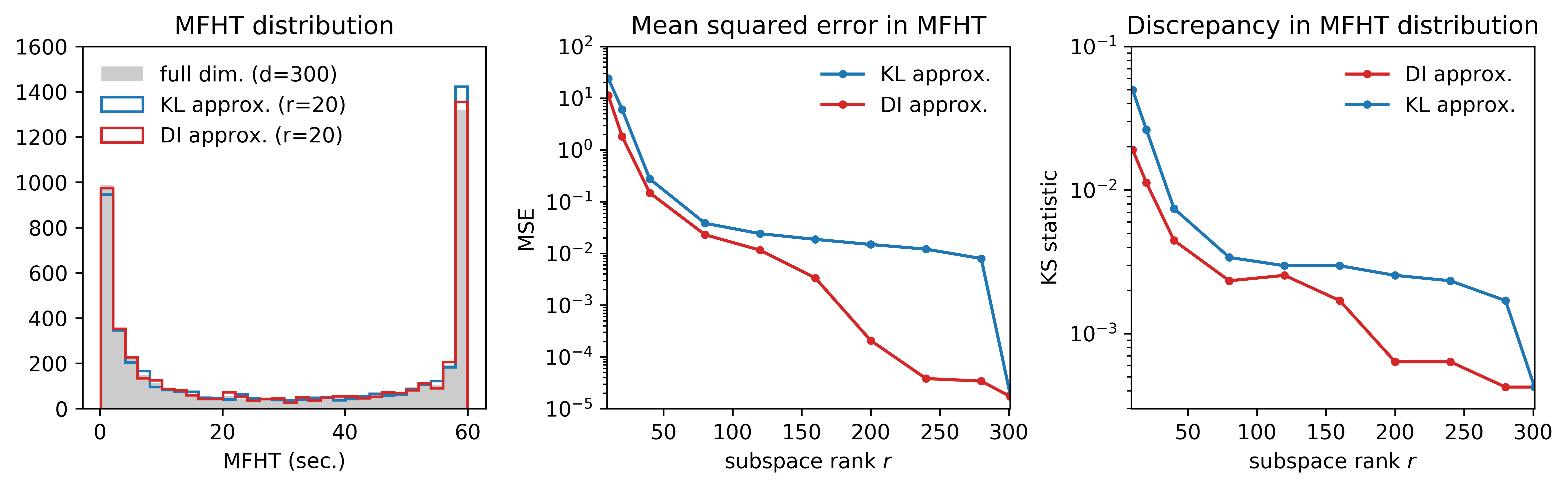}
    \caption{Left: Distribution of the rank-$r$ KL and DI approximations of the mean first hitting time (MFHT) of the double well potential with $r=20$. Center: $\mu$-weighted mean squared error (MSE) in the KL approximations (blue) and DI approximations (red) of the observable. Right: Kolmogorov-Smirnov (KS) statistic of the KL approximation (blue) and DI approximation (red) of the pushforward distribution.}
    \label{fig:dw_mfht}
    \end{figure}

Next, we numerically study how error in approximations of the observable using the rank-$r$ derivative-informed subspace depends on the Monte Carlo sample size $m$ used to estimate the sensitivity matrix. For an observable $F(u)$, we denote the rank-$r$ DI approximation of the observable computed from $\hat{\bm{K}}_{lm}$ by $F^m_r(u)$ and compute its value over every pair $(m,r)$ with $m \in \{100,200,400,800, 1600,3200\}$ and $r \in \{10,20,40,80,160,300\}$. Due to computational constraints, the observable is computed from a single realization of $\hat{\bm{K}}_{lm}$, and the mean squared error $\mathbb{E}_\mu[|F - F^m_r|^2]$ is approximated using $10^3$ fixed samples of $u \sim \mu$. The left panel of \Cref{fig:dw_sample} shows that the MSE in the observable is not sensitive to the Monte Carlo sample size past $m \gtrsim 400$, suggesting that the subspace across ranks $r$ is well-captured by $m=O(10^2)$. The right panel of \Cref{fig:dw_sample} shows that for $r \lesssim 80$, the MSE is relatively less sensitive $m$, whereas for $r > 80$ it depends more strongly on $m$. These results suggest for the given study, a modest number of perturbation samples suffices to capture dominant low-rank structure in the observable, while accurately resolving larger subspaces requires more samples. In practice, the sample budget for estimating the sensitivity matrix can be adapted according to the relative change in the subspace approximation for a given rank, and small $m$ is adequate when a low-rank approximation is sufficient.
    \begin{figure}[htbp]
    \centering
    \includegraphics[width=0.72\textwidth]{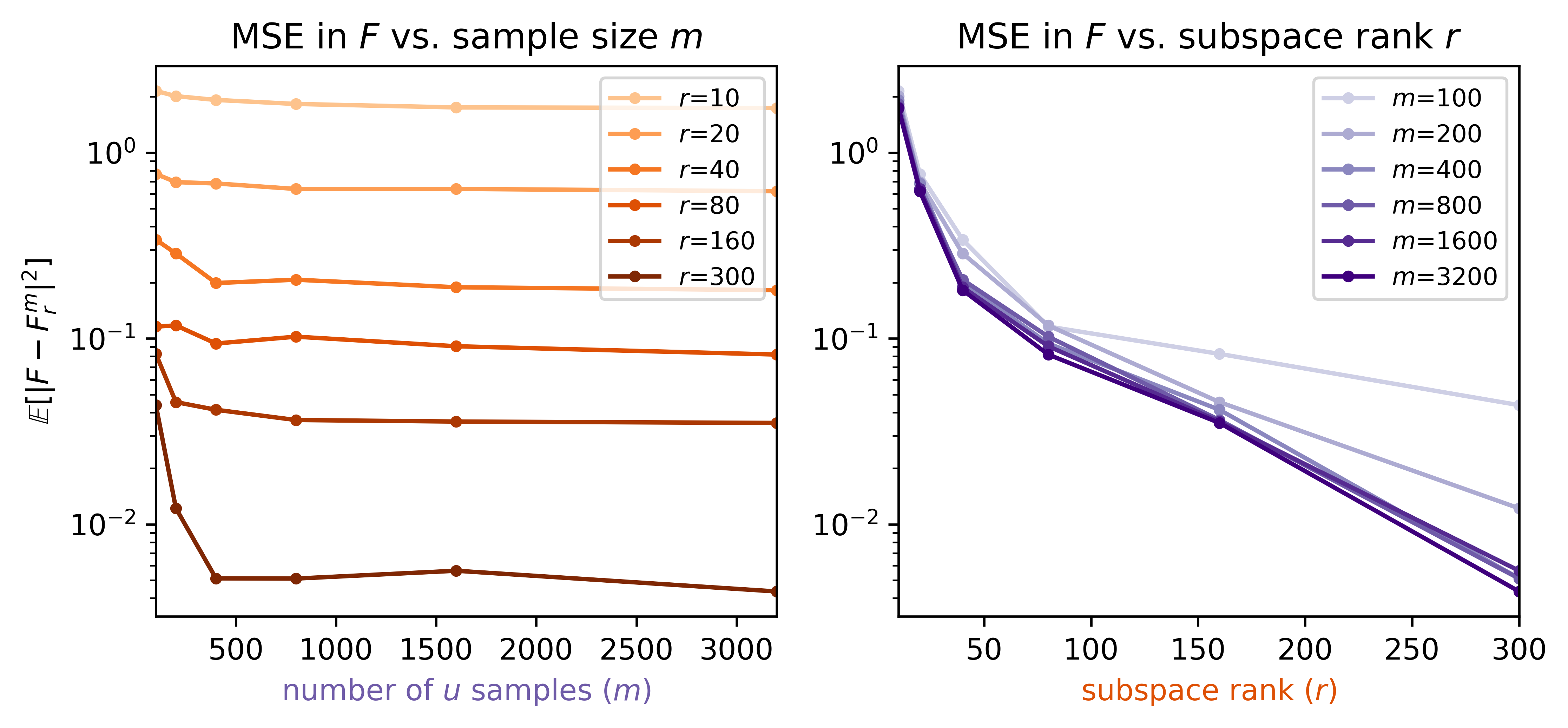}
    \caption{Mean squared error in the observable with respect to the Monte Carlo sample size $m$ used to compute $\hat{\bm{K}}_{lm}$ (left) and with respect to the rank $r$ of the DI subspace used to approximate $F^m_r$ (right). }
    \label{fig:dw_sample}
    \end{figure}
    
Finally, we evaluate the accuracy of the sensitivity matrix and corresponding DI subspace using the control variate estimator of the gradient of the path-space observable. For $l \in \{100,200,400,800\}$ path samples and $m=2.4 \times 10^3$ perturbation samples, we use the control variate-based gradient \eqref{eq:grad_cv} to compute $\hat{\bm{K}}^\textup{CV}_{lm}$ \eqref{eq:egop_cv}. As a point of comparison, we use the standard Monte Carlo gradient \eqref{eq:grad_mc} to compute $\hat{\bm{K}}_{lm}$ \eqref{eq:egop_mc} at the same values of $l$. The reference value of $\bm{K}$ is computed using large-sample Monte Carlo estimators of the gradients with $l=10^4$. \Cref{fig:dw_cv} shows that the control variate formulation leads to two orders of magnitude reduction in the MSE of the estimator of $\bm{K}$ \eqref{eq:mse_k}, reflected in reductions in both the estimator bias and variance. This reduction in MSE carries over to a reduction of error in the DI subspace, measured in terms of weighted subspace distance \eqref{eq:wsd}. 
    \begin{figure}[htbp]
    \centering
    \includegraphics[width=0.95\textwidth]{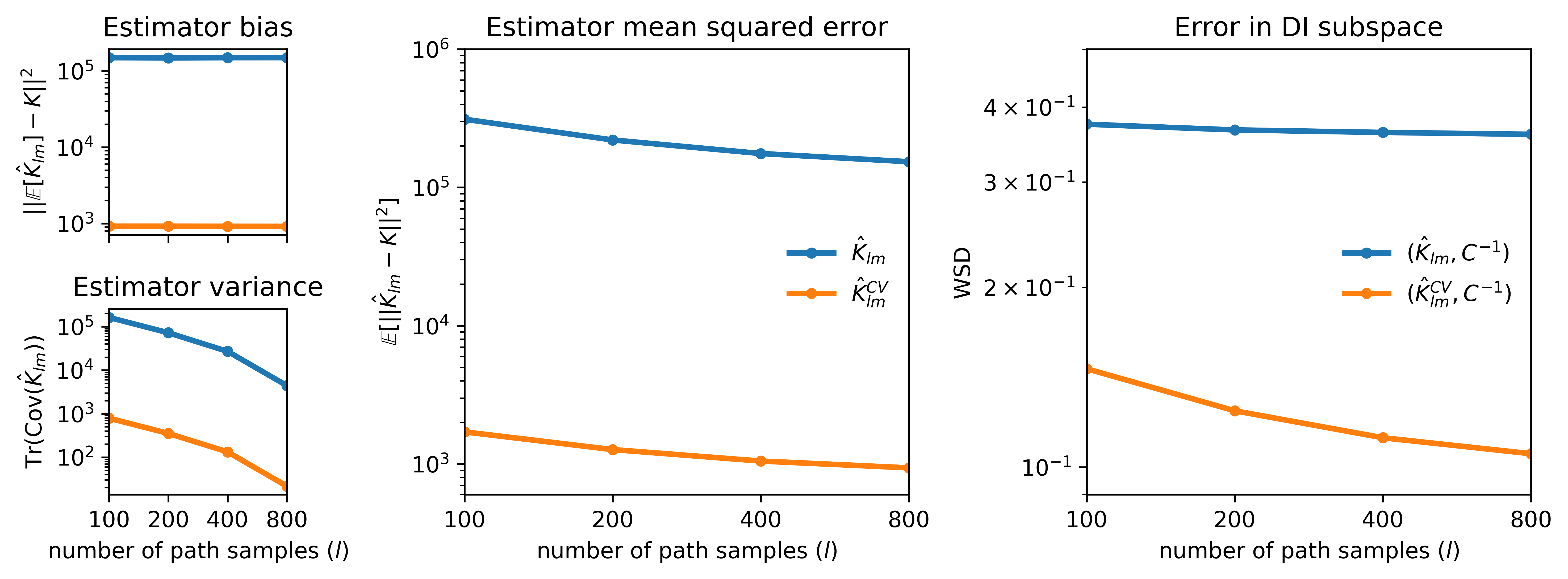}
    \caption{Bias (top left), variance (bottom left), and mean squared error (center) of the Monte Carlo estimator (blue) and control variate estimator (orange) of $\bm{K}$ for a varying number of path samples $l$. Weighted subspace distance (WSD) between the reference and estimated DI subspaces using each estimator of $\bm{K}$ (right).}
    \label{fig:dw_cv}
    \end{figure}

\subsection{M\"uller-Brown potential}
\label{sec:num_mb}

We conduct a similar study with the tempered M\"uller-Brown potential \cite{MullerBrown1979} on $\mathbb{R}^2$, where the potential is defined as
\begin{equation}
\label{eq:mb}
V(x,y) = \sum_{i=1}^{4} A_i \exp\left(a_i (x - x_i^0)^2 + b_i (x - x_i^0)(y - y_i^0) + c_i (y - y_i^0)^2\right)
\end{equation}
\begin{align*}
A &= (-10,\,-5,\,-8.5,\,0.75), & x^0 &= (1,\,0,\,-0.5,\,-1), & y^0 &= (0,\,0.5,\,1.5,\,1), \\ a &= (-1,\,-1,\,-6.5,\,0.7), & b &= (0,\,0,\,11,\,0.6), & c &= (-10,\,-10,\,-6.5,\,0.7).
\end{align*}
We rotate the M\"uller-Brown potential by $-\frac{\pi}{4}$ to align the dominant features with the finite element mesh over the spatial domain. The target path-space observable is a mean truncated first time defined by a maximum time $T=200$, initial condition $x_0 = [-1.42, 0.65]$ located within the primary well of the potential function, and an ellipsoidal domain $D$ about an intermediate well. The M\"uller-Brown potential, initial condition, hitting domain, and a sample hitting path are illustrated in the left panel of \Cref{fig:mb_proj}. 

For discretization of the dimension reduction problem, we define a custom finite element mesh on a bounded domain defined by a level set of the potential function, where the domain boundary is determined by a prescribed energy threshold. The resulting mesh has 2,998 nodes. The Gaussian measure is defined by the Laplacian operator \eqref{eq:laplacian} for $\zeta = 2$ and $\rho=0.5$. 

    \begin{figure}[htbp]
    \centering
    \includegraphics[width=\textwidth]{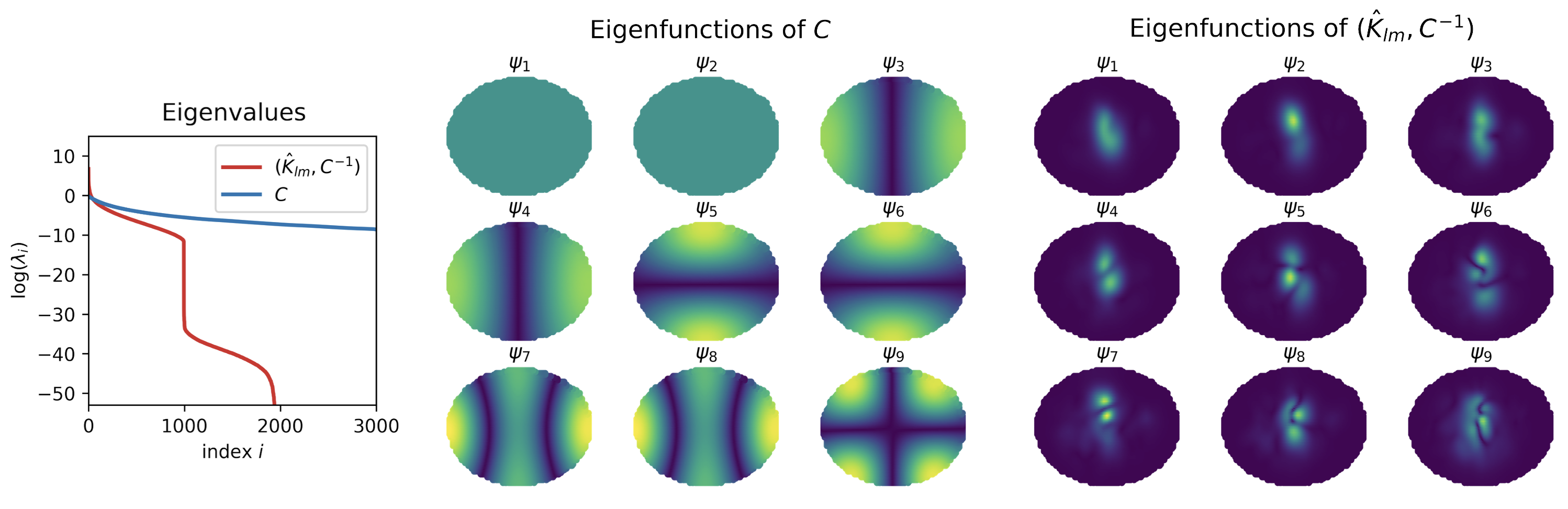}
    \caption{Left: Eigenvalues of $(\bm{K},\bm{C}^{-1})$, corresponding to the DI basis, and of $\bm{C}$, corresponding to the KL basis. Center: First nine eigenfunctions of $\bm{C}$. Right: First nine eigenfunctions of $(\bm{K},\bm{C}^{-1})$.}
    \label{fig:mb}
    \end{figure}

We derive the DI basis of the MFHT observable of the M\"uller-Brown potential, as well as the KL basis, from the spectral decomposition of $(\hat{\bm{K}}_{lm},\bm{C}^{-1})$ and $\bm{C}$, respectively. We observe a sharp decrease in the eigenvalues of the DI basis in the left panel of \Cref{fig:mb}. In the center panel, the leading KL eigenfunctions exhibit ordered spatial variation across the domain, as they capture directions of maximal variance in the perturbation field. In the bottom right panel, the leading DI eigenfunctions instead have concentrated variation about the center region of the domain where the primary well of the potential energy and the initial condition are defined, suggesting that these features are prominent in influencing variation in the observable. 

In \Cref{fig:mb_proj}, we study the projections of a sample drift perturbation $u \sim \mu$ onto the KL basis and DI basis for a subspace dimension $r=50$. We overlay $x_0$ and $D$ on the domain for reference. The DI projection, $u^{\textup{DI}}_r=P^{\textup{DI}}_ru$, highlights features of the potential in regions which are visited the most frequently by the SDE path. The KL projection, on the other hand, smooths the perturbation field but remains agnostic to the SDE path, its initial condition, and its observable. 
    \begin{figure}[htbp]
    \centering
    \includegraphics[width=\textwidth]{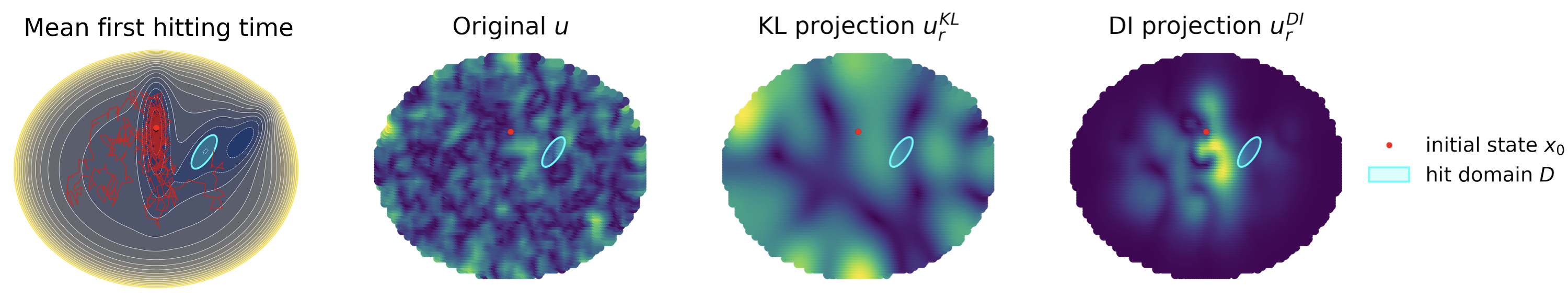}
    \caption{From left to right: The rotated M\"uller-Brown potential of \eqref{eq:mb}. The path-space observable is the mean first time within $[0,T]$ that paths starting at $X_0=x_0$ hit the domain $D$. A sample drift perturbation field $u \sim \mu$ and its projection onto a $50$-dimensional subspace of the KL basis and DI basis. }
    \label{fig:mb_proj}
    \end{figure}

Next, we evaluate the KL and DI approximations of the MFHT observable. The left panel of \Cref{fig:mb_mfht} shows that we are able to capture the pushforward distribution on a significantly reduced-dimensional subspace of $\mathbb{U}$ using the DI basis, reducing from $d=2998$ to $r=10$. The pushforward distribution computed using the KL subspace of the same dimension shows much greater discrepancy to the reference distribution. This discrepancy is further reflected in the center and right panels of \Cref{fig:mb_mfht}, which show that at low subspace ranks $r \lesssim 100$, the DI approximation leads to up to an order of magnitude reduction in the mean squared error (MSE) in $F$ and the Kolmogorov-Smirnov statistic compared to the KL approximation.
    \begin{figure}[htbp]
    \centering
    \includegraphics[width=\textwidth]{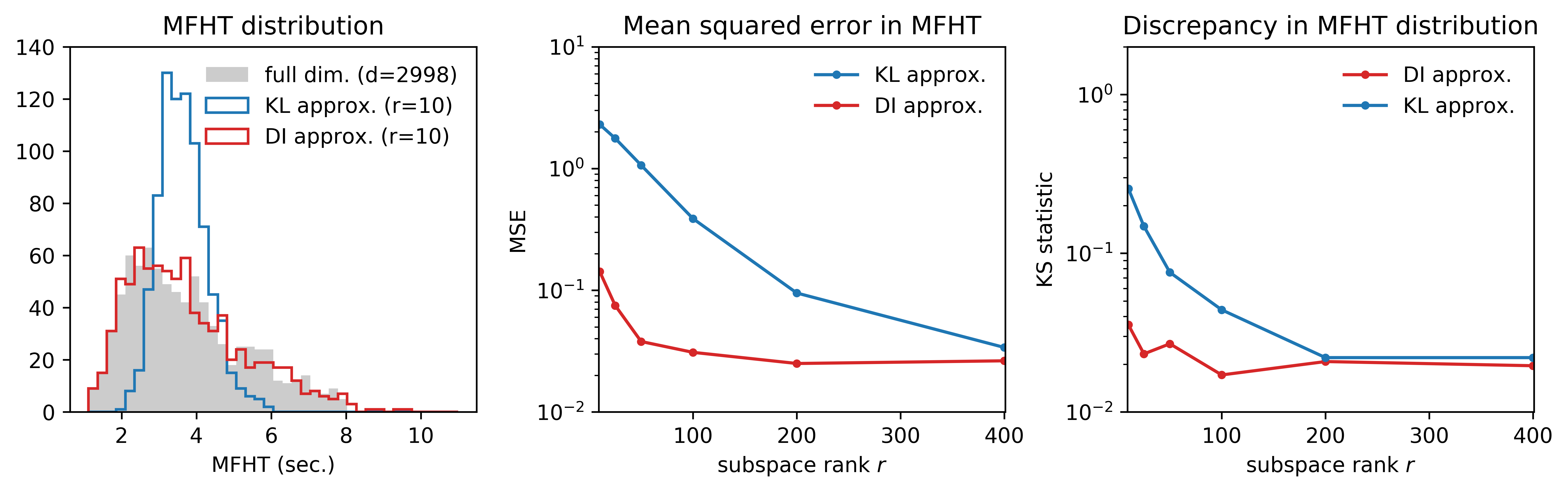}
    \caption{Left: Distribution of the rank-$r$ KL and DI approximations of the mean first hitting time (MFHT) of the M\"uller-Brown potential with $r=10$. The reference pushforward distribution $F_\sharp \mu$ using $u \sim \mu$ in full dimensions is shown in gray. Center: $\mu$-weighted mean squared error (MSE) in the KL approximations (blue) and DI approximations (red) of the observable. Right: Kolmogorov-Smirnov (KS) statistic of the KL approximation (blue) and DI approximation (red) of the pushforward distribution.}
    \label{fig:mb_mfht}
    \end{figure}

\section{Conclusions}
\label{sec:conclusions}

We introduce a derivative-informed dimension reduction approach for the uncertainty quantification of path-space observables of SDEs with uncertain drift functions. We derive the derivative-informed basis from the spectral decomposition of a path-space sensitivity operator involving the Fr\'echet derivative of the observable with respect to the drift perturbation, for which we have an analytical formula in \Cref{lem:frechet}. In \Cref{cond:pathfunc}, we specify sufficient conditions on the path functional that lead to the results that i) the eigenfunctions form an orthonormal basis of the Cameron-Martin space associated with the Gaussian measure on drift perturbations, following \Cref{lem:traceclass}; ii) the eigenvalues are a measure of sensitivity of the observable along the eigendirections, according to \Cref{lem:directderiv}; and iii) the derivative-informed projection operators lead to low-rank approximations of the observable with error guarantees, as shown in \Cref{prop:di}. This construction can be viewed as an infinite-dimensional formulation of active subspaces \cite{Constantine2020} tailored to SDE observables.

We discuss the practical computation of derivative-informed subspaces by introducing a finite element discretization scheme in \Cref{sec:fe} and a control variate for reduced-variance approximation of the sensitivity operator in \Cref{sec:controlvariate}. Our choice of control variate has the advantages that it has an analytical mean and contributes negligible additional costs to estimating the sensitivity operator, as it is already computed in the Fr\'echet derivative formula.

Our numerical examples, conducted using FEniCS \cite{Alnaes2015} and hIPPYlib \cite{Villa2021}, illustrate the advantages of derivative-informed dimension reduction. For mean first hitting time observables in both a double-well potential and the M\"uller-Brown potential, low-rank derivative-informed approximations more accurately reproduce the observable and its pushforward distribution compared to Karhunen-Lo\`eve (KL) approximations of the same rank. These improvements are expected to be especially pronounced for higher state dimensions. While KL modes are agnostic to variation in the observable, the derivative-informed modes preserve accuracy in regions of the state space that are most relevant to the observable, such as the neighborhoods of metastable states, hitting domains, and boundaries. In this sense, the derivative-informed basis acts as a feature extractor: the leading modes correspond to observable-adapted features of the drift function that govern the observable and can serve as inputs to downstream surrogate models of the parameter-to-observable map. 

A key feature of our dimension reduction framework is that it applies to observables defined on path-space, such as first passage times and transition rates between metastable states. Koopman-based reduction techniques, which apply to observables defined on deterministic time intervals, do not directly extend to path-space observables that depend on random stopping times. Our numerical results show that these path-space observables nevertheless possess low-dimensional structure with respect to the drift which is revealed by the derivative-informed subspace. 

There are several possible extensions of this work. One direction is to evaluate the method on higher-dimensional state spaces using alternatives to finite elements to discretize the associated operators -- for instance, random Fourier features or a kernel basis with appropriate constructions of Gaussian measures. Separately, the derivative-informed subspace framework could be integrated with MOR methods for state space reduction. Another direction is to extend the method to consider uncertainties in diffusion coefficients, initial distributions, or boundary conditions and derive the appropriate derivative information. Finally, extensions of the analysis to a broader class of path-space observable, including those corresponding to rare events or with unbounded stopping times, would broaden the scope of this approach in applications to molecular dynamics, stochastic control, and data-driven modeling of stochastic systems.


\appendix

\section{Proof of \Cref{lem:traceclass}}
\label{app:sensitivitypf}

\begin{proof}
    First, we show that $\mathcal{K}|_\mathcal{E}$ is self-adjoint. By definition \eqref{eq:egop_e}, for $w, w' \in \mathcal{E}$,
    \[
    \la \mathcal{K}|_\mathcal{E} w, w' \ra_\mathcal{E} = \mathbb{E}_{u \sim \mu} \big[ \la D_\mathcal{E}F(u)^* D_\mathcal{E}F(u) (w), w' \ra_\mathcal{E} \big] = \mathbb{E}_{u \sim \mu} \big[ \la D_\mathcal{E}F(u) (w), D_\mathcal{E}F(u) (w') \ra_\mathcal{Y} \big] \, .
    \]
    Since the inner product is symmetric in its arguments, we have that $\la \mathcal{K}|_\mathcal{E} w, w' \ra_\mathcal{E} = \la w, \mathcal{K}|_\mathcal{E} w' \ra_\mathcal{E}$.

    Next, since $\la \mathcal{K}|_\mathcal{E} w, w \ra_\mathcal{E} = \mathbb{E}_{u \sim \mu} \big[ \lVert  D_\mathcal{E}F(u)(w)  \rVert^2_\mathcal{Y} \big] \geq 0$ for all $w \in \mathcal{E}$, $\mathcal{K}|_\mathcal{E}$ is positive semidefinite. 

    Finally, we show that $\mathcal{K}|_\mathcal{E}$ is trace-class. For any orthonormal basis $(e_i)_{i \geq 0}$ of $\mathcal{E}$,
    \[
    \text{Tr}(\mathcal{K}|_\mathcal{E}) = \sum_{i \geq 0} \la \mathcal{K}|_\mathcal{E} e_i, e_i \ra_\mathcal{E} = \sum_{i \geq 0} \mathbb{E}_{u \sim \mu} [ \lVert  D_\mathcal{E} F(u) (e_i) \rVert^2_\mathcal{Y}] = \mathbb{E}_{u \sim \mu} \big[ \lVert  D_\mathcal{E} F(u)  \rVert^2_{\text{HS}(\mathcal{E},\mathcal{Y})} \big] \ ,
    \]
    where the interchange of the sum and expectation is valid since all terms are non-negative and the final equality follows from the definition of the Hilbert-Schmidt norm.

    Since \Cref{cond:f} implies that for all $u \in \mathbb{U}$, $DF(u) \in \mathcal{L}(\mathbb{U},\mathcal{Y})$ and the inclusion map $\imath: \mathcal{E} \to \mathbb{U}$ is in $\textup{HS}(\mathcal{E}, \mathbb{U})$ \cite[Lem. 46]{Luo2025}, then the composition $D_\mathcal{E}F(u) = DF(u) \circ \imath$ is in $\textup{HS}(\mathcal{E}, \mathcal{Y})$ by \cite[Prop. A.2.10]{Bogachev1998}. Therefore, for any $u \in \mathbb{U}$, 
    \begin{equation}
        \lVert D_\mathcal{E}F(u) \rVert_{\textup{HS}(\mathcal{E}, \mathcal{Y})} = \lVert DF(u) \circ \imath \rVert_{\textup{HS}(\mathcal{E}, \mathcal{Y})} \leq \lVert DF(u)\rVert_{\mathcal{L}(\mathbb{U}, \mathcal{Y})} \lVert \imath \rVert_{\textup{HS}(\mathcal{E}, \mathbb{U})} \, .
    \end{equation}
    Thus,
    \begin{align}
        \label{eq:l2hs}
        \int_{\mathbb{U}} \lVert D_\mathcal{E}F(u) \rVert^2_{\textup{HS}(\mathcal{E}, \mathcal{Y})} \mu(\upd u) & \leq \int_{\mathbb{U}} \lVert DF(u)\rVert^2_{\mathcal{L}(\mathbb{U}, \mathcal{Y})} \lVert \imath \rVert^2_{\textup{HS}(\mathcal{E}, \mathbb{U})} \mu(\upd u) \\
        & = \lVert \imath \rVert^2_{\textup{HS}(\mathcal{E}, \mathbb{U})} \int_{\mathbb{U}} \lVert DF(u)\rVert^2_{\mathcal{L}(\mathbb{U}, \mathcal{Y})} \mu(\upd u) \ ,
        \notag
    \end{align}
    and the integral in the final line of \eqref{eq:l2hs} is finite by \Cref{cond:f}. Therefore, $D_\mathcal{E} F \in L^2(\mu, \textup{HS}(\mathcal{E}, \mathcal{Y}))$ and $\mathbb{E}_\mu \big[ \lVert  D_\mathcal{E} F  \rVert^2_{\text{HS}(\mathcal{E},\mathcal{Y})} \big] < \infty$. 
    
    Since $\mathcal{K}|_\mathcal{E}$ is self-adjoint and positive semidefinite, $\text{Tr}(\mathcal{K}|_\mathcal{E}) < \infty$ is a necessary and sufficient condition for the operator to be trace-class.
\end{proof}

\section{Proof of \Cref{lem:gen_eig}}
\label{app:adjointpf}

\begin{proof}
    To start, we prove that for strictly positive definite covariance $\mathcal{C}$ and $F$ that satisfies \Cref{cond:f}, $D_\mathcal{E} F(u)^* = \mathcal{C} \, DF(u)^*$. Since $DF(u)$ for $u \in \mathbb{U}$ is assumed to be a bounded operator, its adjoint $DF(u)^*: \mathcal{Y} \to \mathbb{U}$ exists and satisfies  
    \[
        \la DF(u)(v), y \ra_\mathcal{Y} = \la v, DF(u)^*(y) \ra_{\mathbb{U}} \quad \forall v \in \mathbb{U}, y\in\mathcal{Y} \ .
    \]
    Let $h_\mathcal{E} \coloneqq \mathcal{C}^{\half}h \in \mathcal{E}$ for some $h \in \mathbb{U}$. Since $\mathcal{C}^{\half}$ is self-adjoint in the $\mathbb{U}$-inner product,
    \begin{equation}
        \label{eq:cadjoint}
        \la DF(u)(h_\mathcal{E}), y \ra_\mathcal{Y} = \la \mathcal{C}^{\half} h, DF(u)^*(y) \ra_\mathbb{U} = \la h, \mathcal{C}^{\half} DF(u)^*(y) \ra_{\mathbb{U}} \ .
    \end{equation}
    We can also write $h = \mathcal{C}^{-\half} h_\mathcal{E}$ and $\mathcal{C}^{\half} DF(u)^*(y) = \mathcal{C}^{-\half} \big( \mathcal{C} DF(u)^*(y)\big)$. By the definition of the $\mathcal{E}$-inner product, \eqref{eq:cadjoint} becomes
    \[
    \la h, \mathcal{C}^{\half} DF(u)^*(y) \ra_{\mathbb{U}} = \la \mathcal{C}^{-\half} h_\mathcal{E}, \mathcal{C}^{-\half} \big( \mathcal{C} DF(u)^*(y) \big) \ra_{\mathbb{U}} = \la h_\mathcal{E}, \mathcal{C} DF(u)^*(y) \ra_{\mathcal{E}} \ .
    \]
    
    Let $\imath: \mathcal{E} \to \mathbb{U}$ be the inclusion map of an element in the Cameron-Martin space. Since $D_\mathcal{E}F(u)$ is the restriction of $DF(u)$ to directions in $\mathcal{E}$, e.g., $D_\mathcal{E}F(u) = DF(u) \circ \imath$, we have that $DF(u)(w) = D_\mathcal{E}F(u)(w)$ for all $w \in \mathcal{E}$ and 
    \begin{equation}
        \label{eq:cmadjoint}
        \la DF(u)(h_\mathcal{E}), y \ra_\mathcal{Y} = \la D_{\mathcal{E}}F(u)(h_\mathcal{E}), y \ra_\mathcal{Y} = \la h_\mathcal{E}, D_{\mathcal{E}}F(u)^*(y) \ra_{\mathcal{E}} \quad \forall y\in\mathcal{Y} \ .
    \end{equation}
    Equating \eqref{eq:cadjoint} and \eqref{eq:cmadjoint}, we have $\la h_\mathcal{E}, D_{\mathcal{E}}F(u)^*(y) \ra_{\mathcal{E}} = \la h_\mathcal{E}, \mathcal{C} DF(u)^*(y) \ra_{\mathcal{E}}$. By uniqueness of the adjoint, $D_\mathcal{E} F(u)^* = \mathcal{C} \, DF(u)^*$. 

    It then follows that for $w \in \mathcal{E}$,
    \[
    \mathcal{K}|_\mathcal{E}(w) \coloneqq \mathbb{E}_{u \sim \mu}\big[ D_\mathcal{E}F(u)^*  D_\mathcal{E}F(u)(w)  \big] = \mathbb{E}_{u \sim \mu}\big[ \mathcal{C}DF(u)^*  DF(u)(w)  \big] = \mathcal{C}\mathcal{K}.
    \]
    Since $\mathcal{K}$ is an operator on $\mathcal{U}$ and $\mathcal{E} \subseteq \mathbb{U}$, then $\mathcal{K}$ is also an operator on $\mathcal{E}$; similarly for $\mathcal{C}$. It follows that the eigenvalue problem $\mathcal{C}\mathcal{K} \psi^{\textup{DI}}_i=\lambda^{\textup{DI}}_i \psi^{\textup{DI}}_i$ on $\mathcal{E}$ can be premultiplied by $\mathcal{C}^{-1}$ to obtain \eqref{eq:geneig}.
\end{proof}

\section{Proof of \Cref{lem:directderiv}}
\label{app:ddpf} 

\begin{proof}
    Given that $F$ satisfies \Cref{cond:f}, \Cref{lem:traceclass} implies that $\mathcal{K}|_\mathcal{E}$ is a positive, self-adjoint, trace-class operator on $\mathcal{E}$. Hence, its eigenfunctions $\{\psi^{\textup{DI}}_i\}_{i \geq 0}$ form an orthonormal basis of $\mathcal{E}$ with eigenvalues ordered as $\lambda^{\textup{DI}}_0 \geq \lambda^{\textup{DI}}_1 \geq ... \geq 0.$
    
    For $i \geq 0$, consider an arbitrary element $w \in \mathcal{W}_i$. Expanding $w$ in the eigenbasis, we have
        \[
        w = \sum_{j \geq 0} c_j \psi^{\textup{DI}}_j \, , \qquad 
        c_j = \la w, \psi^{\textup{DI}}_j \ra_\mathcal{E} \, , \qquad \sum_{j \geq 0} c_j^2 = \| w \|^2_\mathcal{E} = 1 \, .
        \]
    The orthogonality constraint defining $\mathcal{W}_i$ implies that $c_j=0$ for $j=0,...,i-1$. Therefore, the sums can be equivalently expressed over indices $j \geq i$. We furthermore have
        \[
        \mathcal{K}|_\mathcal{E} w = \sum_{j \geq i} c_j \mathcal{K}|_\mathcal{E} \psi^{\textup{DI}}_j = \sum_{j \geq i} c_j \lambda^{\textup{DI}}_j \psi^{\textup{DI}}_j \, .
        \]
    Using orthonormality of $\{\psi^{\textup{DI}}_i\}_{i \geq 0}$, we have the inner product
        \[
        \la \mathcal{K}|_\mathcal{E} w, w \ra_\mathcal{E} = \sum_{j \geq i} \la c_j \lambda^{\textup{DI}}_j \psi^{\textup{DI}}_j, c_j \psi^{\textup{DI}}_j \ra_\mathcal{E} = \sum_{j \geq i} \lambda^{\textup{DI}}_j c_j^2 = \mathbb{E}_{u \sim \mu} \big[ \| D_\mathcal{E}F(u)(w) \|^2_\mathcal{Y} \big] \, .
        \]
    Since the eigenvalues are in order of decreasing value, $\sum_{j \geq i} \lambda^{\textup{DI}}_j c_j^2 \leq \lambda^{\textup{DI}}_i \sum_{j \geq i} c_j^2 = \lambda^{\textup{DI}}_i$. Equality is attained if and only if $c_i = 1$ and $c_j=0$ for $j \neq i$, which corresponds to $w = \psi^{\textup{DI}}_i$.  Hence, 
        \[
        \lambda^{\textup{DI}}_0 = \underset{\| w \|_\mathcal{E} = 1}{\textup{max}} \ \mathbb{E}_{u \sim \mu} \big[ \| D_\mathcal{E}F(u)(w)\|^2_\mathcal{Y} \big] = \mathbb{E}_{u \sim \mu} \big[ \| D_\mathcal{E}F(u)(\psi^{\textup{DI}}_0)\|^2_\mathcal{Y} \big] \, .
        \]
\end{proof}

\section{Proof of \Cref{prop:condf}}
\label{app:condfpf} 

First, we state and prove an intermediate result on properties of the first-order Fr\'echet derivative of the path-space observable.

\begin{lemma}[Fr\'echet derivative of the path-space observable is a bounded linear operator]
    \label{lem:bounded}
    Let the assumptions of \Cref{lem:frechet} hold. Then for every $u \in \mathbb{U}$ and $F(u)$ defined as in \eqref{eq:observable}, $DF(u) \in \mathcal{L}(\mathbb{U}, \mathcal{Y})$. 
\end{lemma}
\begin{proof}
    Let $u, v \in \mathbb{U}$ be arbitrary. From the definition of $DF(u)(v)$ in \eqref{eq:frechet}, it follows that
    \begin{equation}
            \label{eq:app1}
            \lVert DF(u)(v) \rVert_\mathcal{Y} = \lVert \mathbb{E}^u[\varphi_\tau M^v_\tau]  \rVert_\mathcal{Y} \leq \mathbb{E}^u \big[ \lVert \varphi_\tau \rVert_\mathcal{Y} \, |M^v_\tau| \big] \leq \lVert \varphi_\tau \rVert_{L^2(\mathbb{P}^u, \mathcal{Y})} \lVert M^v_\tau \rVert_{L^2(\mathbb{P}^u)}.
    \end{equation}
    In addition, one has as a result of It\^o isometry that
    \begin{align}
    \label{eq:app2}
            \lVert M^v_\tau \rVert^2_{L^2(\mathbb{P}^u)} & = \mathbb{E}^u[(M^v_\tau)^2] = \mathbb{E}^u \Big[ \int_0^\tau |v(X^u_s)|^2 \upd s \Big] \\
            & \leq \lVert v \rVert^2_{L^\infty} \mathbb{E}^u[\tau] \leq C_\infty^2 \lVert v \rVert^2_{\mathbb{U}} \mathbb{E}^u[\tau] ,
    \notag
    \end{align}
    where the last inequality follows from $\lVert v \rVert_{L^\infty} \leq C_\infty \lVert v \rVert_{\mathbb{U}}$ for some constant $C_\infty$ as a result of $\mathbb{U}$ being continuously embedded in $L^\infty(\mathbb{R}^n)$. Combining \eqref{eq:app1} and \eqref{eq:app2}, an upper bound to $\lVert DF(u)(v) \rVert_\mathcal{Y}$ is given by
    \begin{equation}
    \label{eq:app3}
        \lVert DF(u)(v) \rVert_\mathcal{Y} \leq C_\infty \lVert \varphi_\tau \rVert_{L^2(\mathbb{P}^u, \mathcal{Y})} \lVert v \rVert_{\mathbb{U}} \mathbb{E}^u[\tau]^{\frac{1}{2}} \, .
    \end{equation}
    Since $\tau$ is assumed to be bounded, $\mathbb{E}^u[\tau] < \infty$. By assumption, $\varphi_\tau \in L^2(\mathbb{P}^u, \mathcal{Y})$. It follows that $\lVert DF(u)(v) \rVert_\mathcal{Y}$ is finite for every $v \in \mathbb{U}$ and therefore $DF(u) \in \mathcal{L}(\mathbb{U}, \mathcal{Y})$. 
\end{proof}

\begin{proof}[Proof of \Cref{prop:condf}]
Since $DF(u)$ is a bounded linear operator for every $u \in \mathbb{U}$ by \Cref{lem:bounded}, it directly follows that $DF(u)$ is continuous and $F \in C^1(\mathbb{U},\mathcal{Y})$, proving the first part of \Cref{cond:f}. To prove the second part, we use the result in \Cref{lem:bounded} 
to observe that 
\begin{equation}
    \label{eq:app4}
    \begin{aligned}
    \int_\mathbb{U} \lVert DF(u) \rVert^2_{\mathcal{L}(\mathbb{U}, \mathcal{Y})} \mu (\upd u) & \leq \int_\mathbb{U} C_\infty^2 \lVert \varphi_\tau \rVert^2_{L^2(\mathbb{P}^u, \mathcal{Y})} \mathbb{E}^u[\tau] \mu (\upd u) \\
    & \leq C_\infty^2 \Big( \int_\mathbb{U} \lVert \varphi_\tau \rVert^4_{L^2(\mathbb{P}^u, \mathcal{Y}) } \mu (\upd u) \Big)^{\frac{1}{2}} \Big( \int_\mathbb{U} \mathbb{E}^u[\tau]^2 \mu (\upd u) \Big)^{\frac{1}{2}} \, .
    \end{aligned}
\end{equation}

Since $\mathbb{E}^u[\tau]$ is bounded, the second integral in the final line of \eqref{eq:app4} is finite. The first integral is finite by assumption. Therefore, $DF(u) \in L^2(\mu, \mathcal{L}(\mathbb{U}, \mathcal{Y}))$. 
\end{proof}




\section{Eigendecomposition of the covariance operator}
\label{app:eigfuncs}

In the case of a 1-D state domain $\mathcal{X} \subseteq \mathbb{R}$, the eigenvalues and eigenfunctions of the covariance operator $\mathcal{C}=\mathcal{A}^{-2}$ can be written analytically as follows. Let $[a,b]$ be a subinterval of $\mathbb{R}$ where $L = b - a > 0$. For $\gamma, \delta > 0$ and $\nu \in \mathbb{Z}$, let $\mathcal{A}^{\nu} = (-\gamma \Delta + \delta I)^\nu$ be an operator acting on $C^2([a,b])$ with homogeneous Neumann boundary conditions. Then for $k \in \mathbb{N}_0$, the eigenvalues and eigenfunctions of $\mathcal{A}^{\nu}$ are given by 
\[
\lambda_k = \Big( \gamma \Big( \frac{k \pi}{L} \Big)^2 + \delta \Big)^\nu, \quad \phi_0(x) = \sqrt{\frac{1}{L}}, \qquad \phi_k(x) = \sqrt{\frac{2}{L}} \cos \Big(\frac{k \pi (x - a)}{L} \Big), \ k \geq 1 \, .
\]

\section{Error metrics}
\label{app:errormetrics}

Let $F_r$ denote a rank-$r$ subspace approximation of the path-space observable $F$ for some $r > 0$. The $L^2(\mu)$-weighted error 
\[
\text{MSE}_F(F_r) \coloneqq \mathbb{E}_{u \sim \mu}\big[ \lVert F(u) - F_r(u) \rVert^2\big]
\]
is referred to as the \textit{mean squared error} in the path-space observable.

Denote the empirical cumulative density function (CDF) of the pushforward distribution $F_\sharp \mu$  by
\[
\Pi^N(y) \coloneqq \frac{1}{N}\sum_{i=1}^N \bm{1}\{F(u^{(i)}) \leq y\}, \quad u^{(i)} \sim \mu
\]
and, similarly, the empirical CDF of $(F_r)_\sharp \mu$ by $\Pi^N_r(y)=\frac{1}{N}\sum_{i=1}^N \bm{1}\{F_r(u^{(i)}) \leq y\}$. The \textit{two-sample Kolmogorov-Smirnov (KS) statistic}, 
\[
\mathcal{T}_{\text{KS}} \coloneqq \underset{y \in \mathbb{R}}{\text{sup}} | \Pi^N(y) - \Pi_r^N(y)|
\]
is a measure of discrepancy between the $n$-sample empirical distributions representing $F_\sharp \mu$ and $(F_r)_\sharp \mu$.

Let $\hat{\bm{K}}$ be a sample-based estimator of the sensitivity matrix $\bm{K}$ \eqref{eq:egop} computed from $\mu$-i.i.d. realizations of the perturbation field. Its mean squared error with respect to the Frobenius norm,
\begin{equation}
\label{eq:mse_k}
\text{MSE}_{\bm{K}}(\hat{\bm{K}}) \coloneqq \mathbb{E}\big[ \lVert \bm{K} - \hat{\bm{K}} \rVert^2_F \big] = \lVert \bm{K} - \mathbb{E}[\hat{\bm{K}}] \rVert^2_F + \mathbb{E}\big[\lVert \hat{\bm{K}} - \mathbb{E}[\hat{\bm{K}}]\rVert_F^2\big] \, ,
\end{equation}
admits a bias-variance decomposition, where the expectations are with respect to the sampling distribution of $\hat{\bm{K}}$. The variance term can be written equivalently as $\mathbb{E}\big[\lVert \hat{\bm{K}} - \mathbb{E}[\hat{\bm{K}}]\rVert_F^2\big] = \text{Tr}(\text{Cov}(\text{vec} \,\hat{\bm{K}}))$. 

The \textit{weighted subspace distance} \cite[Def. 2]{Cui2014} gives a measure of discrepancy between two subspaces; here, we apply it to estimates of the derivative-informed subspace. Let $\{(\lambda^{(p)}_j, \bm{\psi}^{(p)}_j)\}_{j=1}^J$ be the eigenvalues and eigenvectors of $(\hat{\bm{K}}^{(p)},\bm{C}^{-1})$. Let $\Psi_r^{(p)} \coloneqq (\bm{\psi}^{(p)}_1,...,\bm{\psi}^{(p)}_r)$ and $(D_r^{(p)})_{ij}=\delta_{ij}(\tilde{\lambda}^{(p)}_i)^{\frac{1}{4}}$ be a matrix with diagonal entries 
\[
\tilde{\lambda}^{(p)}_i \coloneqq \frac{\lambda^{(p)}_i}{\sum_{j=1}^r \lambda^{(p)}_j} \quad , i=1,...,r
\]
for some rank parameter $r > 0$. Then the rank-$r$ weighted subspace distance is 
\begin{equation}
\label{eq:wsd}
\text{WSD}_r\big((\hat{\bm{K}}^{(p)},\bm{C}^{-1}),(\hat{\bm{K}}^{(q)},\bm{C}^{-1}) \big) \coloneqq \sqrt{1 - \Big|\Big| \Big( \Psi_r^{(p)}D^{(p)}_r \Big)^\top \bm{C} \Big( \Psi_r^{(q)}D^{(q)}_r \Big) \Big|\Big|^2_F} \, .
\end{equation}

\bibliographystyle{siamplain}
\bibliography{references}

@book{Oksendal2013,
    author = {Oksendal, B.},
    title = {{Stochastic Differential Equations: An Introduction with Applications (6th ed.)}},
    publisher = {Springer},
    year = {2013},
    doi = {10.1007/978-3-662-03185-8}
}

@book{RevuzYor1999,
  author    = {Daniel Revuz and Marc Yor},
  title     = {Continuous Martingales and Brownian Motion},
  edition   = {3rd},
  year      = {1999},
  publisher = {Springer},
  volume    = {293},
  doi       = {10.1007/978-3-662-06400-9}
}

@book{Constantine2020,
author = {Constantine, P.},
title = {{Active Subspaces: Emerging Ideas for Dimension Reduction in Parameter Studies}},
publisher = {Society for Industrial and Applied Mathematics},
year = {2020}
}

@book{Pavliotis2014,
    title        = {Stochastic Processes and Applications},
    author       = {Pavliotis, G.},
    year         = {2014},
    volume       = {60},
    publisher    = {Springer}
}

@article{Hartmann2012,
   author = {Carsten Hartmann and Christof Schütte},
   doi = {10.1088/1742-5468/2012/11/P11004},
   issn = {17425468},
   issue = {11},
   journal = {J. Stat. Mech. Theory Exp.},
   month = {11},
   title = {Efficient rare event simulation by optimal nonequilibrium forcing},
   volume = {2012},
   year = {2012},
}

@article{Nusken2021,
  author    = {N{\"u}sken, Nikolas and Richter, Lorenz},
  title     = {Solving high-dimensional {Hamilton--Jacobi--Bellman} {PDEs} using neural networks: perspectives from the theory of controlled diffusions and measures on path space},
  journal   = {Partial Differential Equations and Applications},
  volume    = {2},
  number    = {4},
  pages     = {48},
  year      = {2021},
  publisher = {Springer},
  doi       = {10.1007/s42985-021-00102-x}
}

@article{Zhang2021,
title = {Error bounds of the invariant statistics in machine learning of ergodic {Itô} diffusions},
journal = {Phys. D},
volume = {427},
pages = {133022},
year = {2021},
issn = {0167-2789},
doi = {https://doi.org/10.1016/j.physd.2021.133022},
author = {He Zhang and John Harlim and Xiantao Li},
}

@article{Zou2026,
  author  = {Joanna Zou and Han Cheng Lie and Youssef Marzouk},
  title   = {Goal-oriented learning of stochastic differential equations using error bounds on path-space observables},
  journal = {arXiv preprint arXiv:2603.20467},
  year    = {2026},
  eprint  = {2603.20467},
  archivePrefix = {arXiv},
}

@book{Lelievre2010,
    author = {Lelièvre, Tony and Rousset, Mathias and Stoltz, Gabriel},
    title = {{Free Energy Computations}},
    publisher = {Imperial College Press},
    year = {2010},
    doi = {10.1142/p579},
}

@article{Zahm2020,
  author  = {Olivier Zahm and Paul G. Constantine and Cl{\'e}mentine Prieur and Youssef M. Marzouk},
  title   = {Gradient-Based Dimension Reduction of Multivariate Vector-Valued Functions},
  journal = {SIAM Journal on Scientific Computing},
  volume  = {42},
  number  = {1},
  pages   = {A534--A558},
  year    = {2020},
  doi     = {10.1137/18M1221837}
}

@article{Maurais2023,
  title = 	 {Multi-Fidelity Covariance Estimation in the Log-{E}uclidean Geometry},
  author =       {Maurais, Aimee and Alsup, Terrence and Peherstorfer, Benjamin and Marzouk, Youssef},
  booktitle = 	 {Proceedings of the 40th International Conference on Machine Learning},
  pages = 	 {24214--24235},
  year = 	 {2023},
  volume = 	 {202},
  series = 	 {Proceedings of Machine Learning Research},
  publisher =    {PMLR},
}

@article{Cui2014,
  author    = {Tiangang Cui and James Martin and Youssef M. Marzouk and Ari Solonen and Alexander Spantini},
  title     = {Likelihood-informed dimension reduction for nonlinear inverse problems},
  journal   = {Inverse Problems},
  volume    = {30},
  number    = {11},
  pages     = {114015},
  year      = {2014},
  doi       = {10.1088/0266-5611/30/11/114015},
  publisher = {IOP Publishing}
}

@article{Lie2021,
archivePrefix = {arXiv},
eprint = {2106.09149},
author = {Lie, H. C.},
journal = {arXiv},
title = {{Fr\'echet derivatives of expected functionals of solutions to stochastic differential equations}},
year = {2021},
}

@article{Stuart2010,
   author = {Stuart, A. M.},
   title = {{Inverse problems: a Bayesian perspective}},
   journal={Acta Numerica},
   volume={19},
   pages={451-559},
   year = {2010},
   doi = {10.1017/S0962492910000061}
}

@book{Bogachev1998,
  author    = {Vladimir Bogachev},
  title     = {Gaussian Measures},
  series    = {Mathematical Surveys and Monographs},
  volume    = {62},
  year      = {1998},
  publisher = {American Mathematical Society},
  address   = {Providence, RI},
  isbn      = {978-0-8218-0772-1}
}

@article{Villa2021,
  author    = {Villa, Umberto and Petra, Noemi and Ghattas, Omar},
  title     = {{hIPPYlib: An Extensible Software Framework for Large-Scale Inverse Problems Governed by PDEs: Part I: Deterministic Inversion and Linearized Bayesian Inference}},
  journal   = {ACM Transactions on Mathematical Software},
  volume    = {47},
  issue    = {2},
  year      = {2021},
  number = {16},
  pages     = {1--34},
  doi       = {10.1145/3428447},
}

@article{Luo2025,
  author  = {Luo, Dingcheng and O'Leary-Roseberry, Thomas and Chen, Peng and Ghattas, Omar},
  title   = {Dimension reduction for derivative-informed operator learning: An analysis of approximation errors},
  year    = {2026},
  volume = 27,
  number = 182,
  pages = {1-94},
  journal = {Journal of Machine Learning Research},
}

@book{Nickl2023,
  author    = {Richard Nickl},
  title     = {Bayesian Non-linear Statistical Inverse Problems},
  series    = {Zurich Lectures in Advanced Mathematics},
  publisher = {EMS Press},
  address   = {Berlin},
  year      = {2023},
  doi       = {10.4171/ZLAM/30},
}

@article{BuiThanh2013,
  author  = {Bui-Thanh, Tan and Ghattas, Omar and Martin, James and Stadler, Georg},
  title   = {{A Computational Framework for Infinite-Dimensional Bayesian Inverse Problems, Part I: The Linearized Case, with Application to Global Seismic Inversion}},
  journal = {SIAM Journal on Scientific Computing},
  volume  = {35},
  number  = {6},
  pages   = {A2494--A2523},
  year    = {2013},
  doi     = {10.1137/12089586X}
}

@article{Batzner2022,
   author = {Simon Batzner and Albert Musaelian and Lixin Sun and Mario Geiger and Jonathan P. Mailoa and Mordechai Kornbluth and Nicola Molinari and Tess E. Smidt and Boris Kozinsky},
   journal = {Nature Communications},
   volume = {13},
   number = {2453},
   title = {E(3)-equivariant graph neural networks for data-efficient and accurate interatomic potentials},
   year = {2022},
   doi = {10.1038/s41467-022-29939-5}
}

@article{Musaelian2023,
   author = {Musaelian, A. and Batzner, S. and Johansson, A. and et al.},
   journal = {Nature Communications},
   volume = {14},
   number = {579},
   title = {Learning local equivariant representations for large-scale atomistic dynamics},
   year = {2023},
   doi = {10.1038/s41467-023-36329-y}
}

@article{Psichogios1992,
  title={A hybrid neural network‐first principles approach to process modeling},
  author={Dimitris C. Psichogios and Lyle H. Ungar},
  journal={Aiche Journal},
  year={1992},
  volume={38},
  pages={1499-1511},
}

@INPROCEEDINGS{Johannink2019,
  author={Johannink, Tobias and Bahl, Shikhar and Nair, Ashvin and Luo, Jianlan and Kumar, Avinash and Loskyll, Matthias and Ojea, Juan Aparicio and Solowjow, Eugen and Levine, Sergey},
  booktitle={2019 International Conference on Robotics and Automation (ICRA)}, 
  title={Residual Reinforcement Learning for Robot Control}, 
  year={2019},
  volume={},
  number={},
  pages={6023-6029},
  doi={10.1109/ICRA.2019.8794127}}

@inproceedings{Song2021,
  title     = {Score-Based Generative Modeling through Stochastic Differential Equations},
  author    = {Song, Yang and Sohl-Dickstein, Jascha and
               Kingma, Diederik P. and Kumar, Abhishek and
               Ermon, Stefano and Poole, Ben},
  booktitle = {International Conference on Learning Representations},
  year      = {2021},
}

@INPROCEEDINGS{Sudret2017,
  author={Sudret, Bruno and Marelli, Stefano and Wiart, Joe},
  booktitle={2017 11th European Conference on Antennas and Propagation (EUCAP)}, 
  title={{Surrogate Models for Uncertainty Quantification: An Overview}}, 
  year={2017},
  pages={793-797},
  doi={10.23919/EuCAP.2017.7928679}}

@article{Hostettler2018,
title = {Modeling the Drift Function in Stochastic Differential Equations using Reduced Rank Gaussian Processes},
journal = {IFAC-PapersOnLine},
volume = {51},
number = {15},
pages = {778-783},
year = {2018},
note = {18th IFAC Symposium on System Identification SYSID 2018},
issn = {2405-8963},
doi = {https://doi.org/10.1016/j.ifacol.2018.09.137},
author = {Roland Hostettler and Filip Tronarp and Simo Särkkä},
}

@misc{Bal2026,
      title={A Recursive Polynomial Chaos Evolution Method for Stochastic Differential Equations}, 
      author={Guillaume Bal and Shengbo Ma and Su Zhang and Zhiwen Zhang},
      year={2026},
      eprint={2605.03853},
      archivePrefix={arXiv},
      primaryClass={math.NA},
      url={https://arxiv.org/abs/2605.03853}, 
}

@inproceedings{Gorodetsky2015,
  author    = {Alex A. Gorodetsky and Sertac Karaman and Youssef M. Marzouk},
  title     = {Efficient High-Dimensional Stochastic Optimal Motion Control using Tensor Trains},
  booktitle = {Robotics: Science and Systems XI},
  year      = {2015},
  address   = {Rome, Italy},
  publisher = {Robotics: Science and Systems Foundation},
  doi       = {10.15607/RSS.2015.XI.015},
  url       = {https://www.roboticsproceedings.org/rss11/p15.pdf}
}

@InProceedings{Richter2021,
  title     = {{Solving high-dimensional parabolic PDEs using the tensor train format}},
  author    = {Richter, Lorenz and Sallandt, Leon and N\"usken, Nikolas},
  booktitle = {Proceedings of the 38th International Conference on Machine Learning},
  pages     = {8998--9009},
  year      = {2021},
  volume    = {139},
  series    = {Proceedings of Machine Learning Research},
  month     = {18--24 Jul},
  publisher = {PMLR},
}

@articleInfo{Huang2025,
title = {An operator learning perspective on parameter-to-observable maps},
journal = {Foundations of Data Science},
volume = {7},
number = {1},
pages = {163-225},
year = {2025},
doi = {10.3934/fods.2024037},
author = {Daniel Zhengyu Huang and Nicholas H. Nelsen and Margaret Trautner},
}

@article{Du2026,
  title={Rare Event Analysis via Stochastic Optimal Control},
  author={Du, Yuanqi and He, Jiajun and Zhang, Dinghuai and Vanden-Eijnden, Eric and Domingo-Enrich, Carles},
  journal={arXiv preprint arXiv:2604.13213},
  year={2026}
}

@phdthesis{Wang2008,
  author       = {Wang, Limin},
  title        = {Karhunen-Lo\`eve Expansions and their Applications},
  school       = {London School of Economics and Political Science},
  year         = {2008},
  address      = {London, UK},
  month        = {March},
  url          = {http://etheses.lse.ac.uk/2950/}
}

@article{Lindgren2011,
  author  = {Lindgren, Finn and Rue, H{\aa}vard and Lindstr{\"o}m, Johan},
  title   = {{An explicit link between Gaussian fields and Gaussian Markov random fields: The stochastic partial differential equation approach}},
  journal = {Journal of the Royal Statistical Society: Series B (Statistical Methodology)},
  volume  = {73},
  number  = {4},
  pages   = {423--498},
  year    = {2011},
  doi     = {10.1111/j.1467-9868.2011.00777.x}
}

@book{BrennerScott2008,
  author    = {Brenner, Susanne C. and Scott, L. Ridgway},
  title     = {The Mathematical Theory of Finite Element Methods},
  edition   = {3},
  publisher = {Springer},
  address   = {New York},
  year      = {2008},
  doi       = {10.1007/978-0-387-75934-0}
}

@article{Alnaes2015,
  author  = {Aln{\ae}s, Martin S. and Logg, Anders and Ølgaard, Kristian B. and Rognes, Marie E. and Wells, Garth N.},
  title   = {{Automated Solution of Differential Equations by the Finite Element Method: The FEniCS Project}},
  journal = {Acta Numerica},
  year    = {2015},
  volume  = {24},
  pages   = {1--123},
  doi     = {10.1017/S0962492915000035}
}

@article{MullerBrown1979,
  author    = {Klaus M{\"u}ller and Leo D. Brown},
  title     = {Location of saddle points and minimum energy paths by a constrained simplex optimization procedure},
  journal   = {Theoretica Chimica Acta},
  year      = {1979},
  volume    = {53},
  number    = {1},
  pages     = {75--93},
  doi       = {10.1007/BF00547608},
}

@article{Mezic2005,
  author  = {Mezi{\'c}, Igor},
  title   = {{Spectral Properties of Dynamical Systems, Model Reduction and Decompositions}},
  journal = {Nonlinear Dynamics},
  volume  = {41},
  pages   = {309--325},
  year    = {2005}
}

@article{Williams2015,
  author  = {Williams, Matthew O. and Kevrekidis, Ioannis G. and Rowley, Clarence W.},
  title   = {{A Data--Driven Approximation of the Koopman Operator: Extending Dynamic Mode Decomposition}},
  journal = {Journal of Nonlinear Science},
  volume  = {25},
  pages   = {1307--1346},
  year    = {2015}
}

@article{Klus2018,
  author  = {Klus, Stefan and N{\"u}ske, Feliks and Koltai, P{\'e}ter and Wu, Hao and Kevrekidis, Ioannis and Sch{\"u}tte, Christof and No{\'e}, Frank},
  title   = {Data-Driven Model Reduction and Transfer Operator Approximation},
  journal = {Journal of Nonlinear Science},
  volume  = {28},
  pages   = {985--1010},
  year    = {2018}
}

@article{Klus2020,
  author  = {Klus, Stefan and N{\"u}ske, Feliks and Peitz, Sebastian and Niemann, Jan-Hendrik and Clementi, Cecilia and Sch{\"u}tte, Christof},
  title   = {{Data-Driven Approximation of the Koopman Generator: Model Reduction, System Identification, and Control}},
  journal = {Physica D: Nonlinear Phenomena},
  volume  = {406},
  pages   = {132416},
  year    = {2020}
}

@article{Peitz2019,
  author  = {Peitz, Sebastian and Klus, Stefan},
  title   = {Koopman Operator-Based Model Reduction for Switched-System Control of PDEs},
  journal = {Automatica},
  volume  = {106},
  pages   = {184--191},
  year    = {2019}
}

\end{document}